\documentclass[journal,twoside]{IEEEtran} 

\usepackage{cite}
\usepackage{amsmath,amssymb,amsfonts}
\usepackage{algorithm}
\usepackage{algorithmic}
\usepackage{graphicx}
\usepackage{hyperref}
\hypersetup{hidelinks=true}
\usepackage{textcomp}

\usepackage{amsthm}
\newtheorem{assumption}{Assumption}
\newtheorem{remark}{Remark}
\newtheorem{lemma}{Lemma} 
\newtheorem{theorem}{Theorem}

\usepackage{tikz}
\usetikzlibrary{shapes.geometric, positioning, arrows.meta} 

\def\BibTeX{{\rm B\kern-.05em{\sc i\kern-.025em b}\kern-.08em
    T\kern-.1667em\lower.7ex\hbox{E}\kern-.125emX}}
    
\begin{document}
\title{Consensus-based Distributed Optimization for Multi-agent Systems over Multiplex Networks}
\author{Christian D.~Rodr\'iguez-Camargo,
       Andr\'es F.~Urquijo-Rodr\'iguez,
        and~Eduardo~Mojica-Nava
\thanks{Manuscript received MM DD, YYYY; revised MM DD, YYYY.}
\thanks{C. D. Rodr\'iguez-Camargo is with The Atomic, Molecular, Optical and Positron Physics (AMOPP) group of the Department of Physics and Astronomy from the University College London and Programa de Investigaci\'on sobre Adquisici\'on y An\'alisis de Se\~nales (PAAS-UN) from Universidad Nacional de Colombia. E-mail: christian.rodriguez-camargo.21@ucl.ac.uk}

\thanks{A. F. Urquijo-Rodr\'iguez is with the Industrial Engineering Program and the Centro de Estudios Industriales y Log\'isticos para la Productividad (CEIL, MD) both from the Corporaci\'on Universitaria Minuto de Dios, and Grupo de Superconductividad y Nanotecnolog\'ia and Grupo de \'Optica e Informaci\'on Cu\'antica both from the Department of Physics of Universidad Nacional de Colombia.}

\thanks{E. Mojica-Nava is with the Department of Electrical and Electronics Engineering and Programa de Investigaci\'on sobre Adquisici\'on y An\'alisis de Se\~nales (PAAS-UN) from Universidad Nacional de Colombia, Bogotá, Colombia.
E-mail: eamojican@unal.edu.co}}

\maketitle

\begin{abstract}
Multilayer networks provide a more comprehensive framework for exploring real-world and engineering systems than traditional single-layer networks, consisting of multiple interacting networks. However, despite significant research in distributed optimization for single-layer networks, similar progress for multilayer systems is lacking. This paper proposes two algorithms for distributed optimization problems in multiplex networks using the supra-Laplacian matrix and its diffusion dynamics. The algorithms include a distributed saddle-point algorithm and its variation as a distributed gradient descent algorithm. By relating consensus and diffusion dynamics, we obtain the multiplex supra-Laplacian matrix. We extend the distributed gradient descent algorithm for multiplex networks using this matrix and analyze the convergence of both algorithms with several theoretical results. Numerical examples validate our proposed algorithms, and we explore the impact of interlayer diffusion on consensus time. We also present a coordinated dispatch for interdependent infrastructure networks (energy-gas) to demonstrate the application of the proposed framework to real engineering problems. 
\end{abstract}


\section{Introduction}
\label{sec:introduction}

\IEEEPARstart{D}{uring} the last two decades, the single-layer network representation of complex systems has proven to be a valuable tool for revealing the relationships between topological properties of different networked systems and their dynamics~\cite{barabasi2016network, newman2010networks}. The quantitative study of complex networks has generated a wide set of interesting results in the deeper understanding of complex systems, ranging from application in engineering, and natural sciences to social sciences~\cite{albert2002statistical, dorogovtsev2003evolution}. 
However, as network science advances and more data is gathered from real-world networked systems, there is a growing need to understand real and engineered systems as networks of networks rather than isolated networks~\cite{bianconi2018multilayer}. These systems are commonly referred to as multilayer networks~\cite{kivela2014multilayer}.

 A multilayer network refers to systems composed of several interacting networks~\cite{bianconi2018multilayer, kivela2014multilayer, boccaletti2014structure}.  Multilayer networks were originally presented in social science to describe the various possibilities of connections between nodes in a social network~\cite{boccaletti2014structure}. Today, multilayer networks are being studied in various fields, such as neuroscience~\cite{vaiana2020multilayer}, molecular biology~\cite{kiani2021networks}, ecology~\cite{pilosof2017multilayer}, transportation~\cite{du2016physics}, and power systems~\cite{toro2021multiplex}. Despite there are significant advances in phenomena such as percolation~\cite{santoro2020optimal}, cascades failures~\cite{turalska2019cascading}, diffusion dynamics~\cite{perc2019diffusion}, resilient consensus \cite{shang2020resilient}, adaptive control \cite{guo2021adaptive}, epidemic spreading~\cite{salehi2015spreading}, there has been a lack of progress in generalizing single-layer network distributed optimization problems to multilayer networks. 

In traditional distributed optimization problems, each node typically only has the availability to a convex function, and the objective function is expressed as a sum of these functions~\cite{yang2019survey}. Each agent aims to achieve the optimal solution of the total convex function by exchanging information with its neighboring nodes and performing local computations~\cite{nedic2009distributed, nedic2018distributed}. Numerous techniques have been suggested to address this distributed optimization problem, from discrete-time gradient descent \cite{nedic2009distributed}, \cite{nedic2018distributed} to recent continuous-time approaches \cite{Elia2011control}, \cite{kia2015dynamic}, \cite{lin2016distributed}, \cite{yang2016multi}, \cite{wu2023distributed},  \cite{ma2019novel}. From a continuous-time perspective, the methods based on saddle-point dynamics \cite{cherukuri2017saddle}, \cite{feijer2010stability}  have emerged as an alternative to view optimization algorithms as dynamical systems \cite{colombino2019online}, bringing up brand-new scenarios to deal with complex real-world systems such as complex networks as it is discussed in the recently introduced concept of feedback-based optimization \cite{hauswirth2021optimization}, \cite{feedback2022}.

In this paper, we extend control-based and consensus-based distributed optimization algorithms from single-layer networks to multiplex networks, which generalizes the relationship between consensus and diffusion dynamics in single-layer networks. We introduce the augmented supra-Lagrangian concept, which incorporates implicit gradient tracking  \cite{kia2015dynamic} \cite{nedic2017achieving} for both intralayer and interlayer gradients. A multiplex network is a collection of graphs comprised of $M$ distinct layers, where the same set of nodes $N$ are connected through links associated with $M$ different kinds~\cite{bianconi2018multilayer}. A substantial distinction exists between considering all interactions at the same level and incorporating heterogeneous information of various interactions at different levels. In a multiplex network, each interaction carries a distinct implication, and this characteristic is correlated with other structural features, enabling us to glean more information from the intricate system under study. Typically, the dynamic interactions among nodes in a multiplex network assume varying functional forms based on the nature of the link \cite{bianconi2018multilayer}. The increasing interest in multiplex networks has highlighted the significant impact of their structure on the behavior of dynamic processes. Therefore, it is crucial to study the properties and dynamics of multiplex networks to comprehend the complex relationships and interdependencies that arise in the real world. The variation in diffusion rates across different types of links within multiplex networks fundamentally alters the characteristics of this dynamic process, leading to a range of practical implications. From an engineering perspective, it is anticipated that multiplex networks will be utilized to model and manage energy in multienergy systems. A multienergy system is an integrated energy system that employs multiple sources of energy such as gas and electricity, for instance, in a coordinated and optimized manner to provide efficient, reliable, and sustainable energy services \cite{mancarella2014mes}. The objective is to maximize the use of renewable and low-carbon energy sources, as well as to enhance energy efficiency and reduce greenhouse gas emissions \cite{guelpa2019towards}. Multienergy systems are becoming increasingly significant as a means of satisfying the growing demand for energy while decreasing the environmental impact of energy production and consumption \cite{chertkov2020multienergy}. The proposed framework is applied to real engineering problems by presenting a coordinated dispatch for interdependent infrastructure networks such as energy and gas.

The main contributions of this work are threefold. First, by recognizing the relationship and mathematical similarity between the consensus equation in multi-agent systems, which is utilized to solve distributed optimization problems, and the diffusion equation in statistical mechanics, it has been possible to propose connections between diffusion processes in multiplex networks and consensus-based distributed optimization algorithms. We expand upon these definitions by utilizing the supra-Laplacian matrix of a multiplex network. Secondly, we derive a distributed saddle point algorithm for convex optimization in multiplex networks and a variation of this algorithm; a distributed gradient descent algorithm for multiplex networks using the relationship between the diffusion equation and consensus equation for multiplex networks using the soft-penalty method. Finally, we show the convergence of each node in each layer to the optimal value of the convex objective function. We demonstrate the accuracy of both algorithms. Notably, only knowledge of the neighborhood is required to achieve the local and global optimal value in multiplex networks. Moreover, we investigate the existence of critical phenomena in the consensus time by manipulating the diffusion constants. Several numerical examples are presented to validate the proposed algorithms, and the impact of interlayer diffusion on consensus time is explored. A coordinated dispatch for interdependent infrastructure networks (energy-gas) is finally presented to demonstrate the application of the proposed framework to real engineering problems.

This paper is organized as follows. In Section II, we present the preliminaries of multiplex networks, the diffusion processes, and the control approach of consensus and diffusion dynamics. In Section III, we establish the problem statement of distributed optimization in multiplex networks. In Section IV, we present the main result of this paper and a convergence analysis of our theorem related to the distributed primal-dual saddle-point algorithm for multiplex networks. In Section V we present the generalized gradient descent algorithm. In Section VI numerical examples are presented with an additional real-life example in energy management systems. In addition, we evaluate the existence of critical phenomena in the function of the diffusion parameters. The conclusions are presented in Section VII.

\section{Preliminaries}

\subsection{Consensus-based Optimization}

We start with the basic definitions of graph theory. A graph $G = G(V, E)$ is a finite set of vertices $V$ with a set of edges $E$. A vertex $v$ in $G$ is denoted by either $v \in V$ or $v \in G$. Furthermore, A graph $G$ is connected if a path can be obtained between any two vertices $v_r$ and $v_p$ in the graph, that is, a sequence of vertices $v_r = v_{0} \sim v_{1} \sim v_{2} \sim \cdots \sim v_{n-1} \sim v_{n} = v_p$ such that each pair of consecutive vertices $v_{j-1}$ and $v_{j}$ are connected by an edge for $j = 1, 2, ..., n$. Here, $v_i \sim v_j$ denotes the fact that two vertices $v_i$ and $v_j$ are connected by a link in $E$. An undirected weighted graph is associated with a weight function $\omega: V \times V \rightarrow \mathbb{R}^{+}$ satisfying:
\begin{itemize}
    \item[(i)]$\omega (v_i,v_j)=\omega (v_j,v_i)$, $v_i, v_j \in V$.
    \item[(ii)] $\omega (v_i,v_j)=0$ if and only if $\{ v_i, v_j \}\not\in V$.
\end{itemize}

We denote $\{ v_i, v_j\}$ as the link connecting the vertices $v_i$ and $v_j$. The degree $d_{\omega}(v_i)$ of a vertex $v_i$ is defined to be $d_{\omega}(v_i):=\sum _{j\in \mathcal{N}_i}\omega (v_i,v_j)$, where $\mathcal{N}_i$ is the neighborhood of the node $v_i$. 

For a measure related to the weighted degree of each vertex, the $\omega$-Laplacian operator $\Delta _{\omega}$ is represented as a matrix as
\begin{equation*}
\Delta _{\omega}(v_i,v_j)=
\begin{cases}
1-\frac{\omega _{v_i,v_i}}{d_{\omega}(v_i)}, \quad \text{if}\, v_i=v_j\, \text{and}\, d_{\omega}(v_i)\neq 0,\\
-\frac{\omega (v_i,v_j)}{d_{\omega}(v_i)}, \quad \text{if}\, v_i\sim v_j, \\
0,\quad \text{otherwise}.
\end{cases}
\end{equation*}

The matrix representation of $\Delta _{\omega}$ is equivalent to the Laplacian matrix $L$ given by $L=D^{1/2}\Delta _{\omega}D^{-1/2}$,where $D$ is a diagonal matrix with entries $D(v_i,v_j)=d_{\omega}(v_i)$. $L$ is a non-negative definite symmetric matrix, and its eigenvalues are given by $0<\lambda _{0}\leq \lambda _{1}\leq \lambda _{2}\leq \cdots \leq \lambda _{N-1}$. It is possible to verify that $\lambda _{0}=0$, $\lambda _{1}>0$.

Now, we can introduce the traditional consensus equation for single-layer networks of multiagent systems using the Laplacian matrix. First, consider the dynamic for each node as 

\begin{equation}\label{consenagent}
\dot{x}_{i}(t)= -\sum _{j \in \mathcal{N}_{i}}\omega _{ij}(t)(x_{i}(t)-x_{j}(t)),
\end{equation}
where $x_{i}(t)$ is the information state of an agent in node $i$, $\mathcal{N}_{i}$ is the neighborhood of node $i$, and $\omega _{ij}(t)$ is a positive time-varying weighting factor. Equation \eqref{consenagent} can be presented in compact form as a diffusion equation

\begin{equation} \label{consen}
\dot{x}=-Lx,    
\end{equation}
where $L$ is the graph Laplacian.  Using the consensus equation \eqref{consen}, several continuous-time saddle point dynamics have proposed to solve a distributed optimization problem in multiagent systems \cite{cherukuri2017saddle, feijer2010stability, Elia2011control, lin2016distributed}

\begin{equation}\label{opprob1}
\min _{x \in \mathbb{R}^{n}}\tilde{f}(x)=\sum _{i=1}^{n}f_{i}(x) \quad \text{s.t.} \quad
Lx = 0,
\end{equation}
where $f(x)$ is the global objective function to be minimized subject to the communication constraints between agents represented by the graph Laplacian. The corresponding saddle-point dynamics for optimization problem \eqref{opprob1}

\begin{eqnarray*}
\dot{x} &=& -L x - \frac{\partial \tilde{f}(x)}{\partial x}-L\lambda,\\
\dot{\lambda}&=& Lx,
\end{eqnarray*}
where $\lambda$ is the vector of Lagrange multipliers \cite{bertsekas2009convex}.

Several references have suggested the use of a control approach to interpret and improve diffusion and consensus dynamics in situations where control elements can be employed \cite{dorfler2017distributed, hauswirth2021optimization, colombino2019online, Elia2011control}. By leveraging this approach, new distributed optimization problems can be obtained. For instance, assume that each agent has to find the solution to the following optimization problem
\begin{equation}
\min \sum _{i=1}^{n}f_{i}(y) \quad \text{s.t.} \quad y\in \mathbb{R}^{n},
\label{problemopt001}
\end{equation}
where $f_{i}:\mathbb{R}^{n}\rightarrow \mathbb{R}$ is a convex function which is only available to node $i$. To solve Problem (\ref{problemopt001}), it is proposed a continuous time dynamic model based on driving the state to the optimal solution set~\cite{lin2016distributed}. Assuming that the state of node $i$ is $y_{i}\in \mathbb{R}^{n}$, the dynamics of each node are assumed to be described by the following ordinary differential equations

\begin{equation}\label{dyndist}
\dot{y}_{i}(t)=\sum _{j\in N_{i}}a_{ij}(y_{j}(t)-y_{i}(t)) -g_{i}(y_{i}(t)), 
\end{equation}
where $a_{ij}=a_{ji}>0$ are coefficients associated with node $i$, $g_{i}(y_i(t))$ is the gradient of $f_{i}$ at the point $y_{i}(t)$.

The dynamical system \eqref{dyndist} can be written compactly as
\begin{equation}\label{ccm001}
\dot{y}(t)=-\mathbf{L}y(t)-G(y(t))
\end{equation}
where $\mathbf{L} = L \otimes I_{n}$ being $L$ the Laplacian matrix, and $G(y(t))$ is a succession of gradient functions $g_{i}(y_i(t))$. 

\subsection{Multiplex Networks}
\begin{figure}[t]
\begin{center}
\includegraphics[width=.4\textwidth]{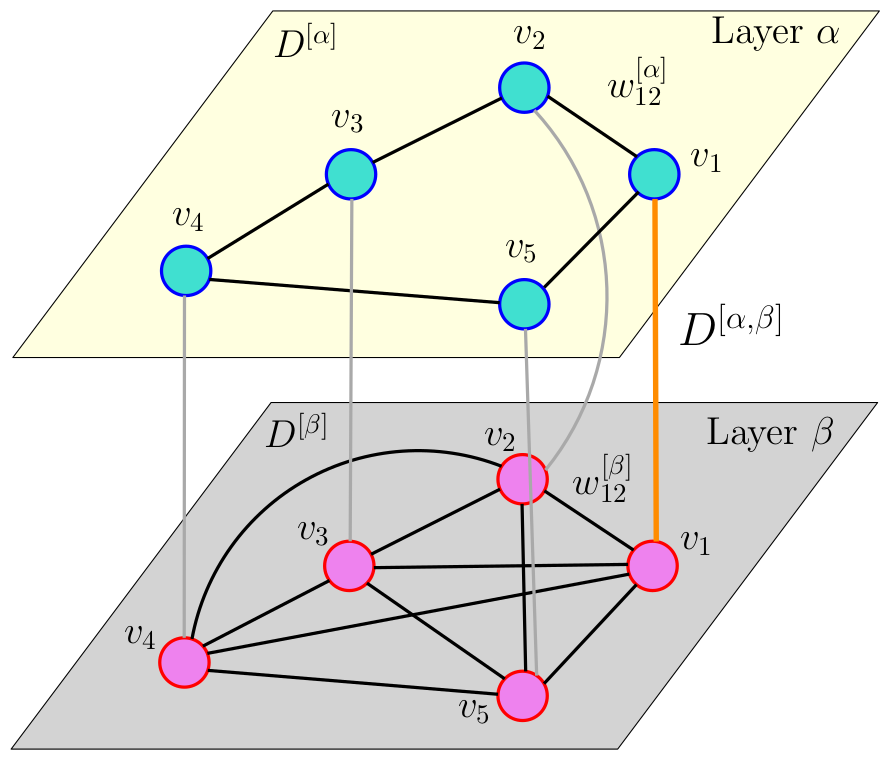}
\end{center}
\caption{Multiplex network with $M=2$ layers and $N=5$ nodes in each layer $\alpha$ and $\beta$ and its intra/inter-layer diffusion constants.}
\label{multi5ddd}
\end{figure}

A multiplex network is a subclass of a multilayer network, where nodes in different layers are mapped in a one-to-one correspondence, and interlinks exclusively connect to corresponding replica nodes~\cite{bianconi2018multilayer}. Multiplex networks are commonly used to represent relationships between a common set of nodes, where each layer corresponds to a specific type of interaction. There are two ways to represent a multiplex network. In the first representation, corresponding replica nodes are not distinguished, and interlinks are not explicitly used. In the alternative representation, the replica nodes are considered discernible agents, and interlinks are explicitly described.

In this paper, we focus on the alternative representation, which includes interlinks and treats corresponding replica nodes as distinguishable entities. This approach allows for a more explicit and detailed description of the multiplex network topology and facilitates the analysis of complex interactions between nodes in different layers~\cite{bianconi2018multilayer}. The multiplex network can be represented as $N$ nodes, denoted by $i=1,2,..., N$, and $M$ layers. Every node $i$ is associated with $M$ \emph{replica node}, $(i, \alpha)$, where $\alpha=1,2,...,M$ represents the identification of node $i$ in layer $\alpha$. We define the set of nodes $V = \{ i \;|\; i \in \{ 1, 2, ..., N \} \}$ and the set of layer $\mathcal{M}=\{ \alpha \;|\; \alpha \in \{ 1, 2, ..., M \} \}$ for the multiplex network. 

A multilayer network can be represented by $M$ node sets $V_{\alpha}$, where each set represents the replicas of the nodes in $V$ for layer $\alpha$. To distinguish between the different replica nodes, the multiplex network is formed by the tuple $(Y, \mathbf{G}, \mathcal{G})$ being $\mathbf{G}$ given by $\mathbf{G} = (G_{1}, G_{2}, ..., G_{\alpha}, ..., G_{M})$, and $G_{\alpha}=(V_{\alpha}, E_{\alpha})$ being the network in layer $\alpha$ and the interactions between different layers in a multiplex network are characterized by an $M\times M$ network $\mathcal{G}$. Each element of the list, denoted as $\mathcal{G}_{\alpha,\beta}$, corresponds to a multiplex network that consists of node sets $V_{\alpha}$ and $V_{\beta}$, and link set $E_{\alpha,\beta}$. Interlinks are defined as links connecting nodes in layer $\alpha$ to nodes in layer $\beta$. In this particular representation, it is exclusively the replica nodes in layer $\alpha$ that are connected to their corresponding replica nodes in layer $\beta$ through the interlinks in $E_{\alpha,\beta}= \{ [(i,\alpha),(i,\beta)]| \; i \in \{ 1, 2, ..., N \}\}$. 

Finally, a supra-adjacency matrix $\mathcal{A}$ with entries $\mathcal{A}_{i\alpha,j\beta}$ and dimension $N\cdot M \times N\cdot M$ can be defined to indicate both the intralinks and the interlinks as
\begin{equation}
\mathcal{A}_{i\alpha,j\beta} = 
\begin{cases}
a_{ij}^{[\alpha]}\quad \text{if} \quad \alpha = \beta, \\
\delta (i,j)\quad \text{if} \quad \alpha \neq \beta,
\end{cases} 
\end{equation}
where $a_{ij}^{[\alpha]}$ are the matrix elements of the $M$ different adjacency matrices $A^{[\alpha]}$ in layer $\alpha$, and $\delta (i,j)$ is defined as the Kronecker delta. For example, entries $a_{ij}^{[\alpha]}$ of an undirected multiplex network with weights are as follows
\begin{equation}
a_{ij}^{[\alpha]} = \begin{cases}
w _{ij}^{[\alpha]}\quad \text{if $i$-node is connected to $j$-node in layer $\alpha$}\\
\; \qquad\text{and weight }\; \; w _{ij}^{[\alpha]}\\
0 \qquad \text{otherwise.}
\end{cases} 
\end{equation}

Notice that multiplex networks can represent interactions among diverse sets of nodes, or they can describe interactions within the same set of nodes, with each layer meaning a distinct type of interaction. In scenarios such as interdependent power grid and communication infrastructures, each power plant relies on a specific node in the communication network for monitoring dynamics, or connectivity between two stations in a city through various transportation routes like bus, train, and metro \cite{buldyrev2010catastrophic}, \cite{bianconi2018multilayer}. A multiplex network with $M=2$ layers and $N=5$ nodes in each layer is shown in Fig. \ref{multi5ddd}. In the next section, it is presented the main concepts of diffusion dynamics on multiplex networks and their different diffusion parameters such as interlayer and intralayer coefficients.

\subsection{Diffusion Processes on Multiplex Networks}

Diffusion processes play a pivotal role in multiplex networks, as they ensure communication between nodes across different layers and within nodes of the same layer. This configuration might represent, for instance, diffusion dynamics occurring over multimodal transportation networks, where individuals diffuse within and between various layers, such as bus, subway, and so on \cite{gomez2013diffusion}. In interdependent networks such as multienergy systems, where a power network and a gas network diffuse energy to produce different types of power, the diffusion process is often influenced by factors such as technological interdependencies, regulatory frameworks, economic considerations, and user behavior. Various authors, such as~\cite{biggs1993algebraic,  curtis1991dirichlet}, have studied the framework of the $\omega$-Laplace equation on graphs, considering a diffusion equation on electric networks and represents a significant problem in the field of network analysis~\cite{gilbert2016diffuse}.

Diffusion processes within each layer and across a different layer can be described by a dynamical diffusion state $y_{i}^{[\alpha]}(t)$ associated with each replica node $(i,\alpha)$ of the multiplex network with $i=1,2,...,N$ and $\alpha = 1,2,...,M$. We assume that both intralinks and interlinks can sustain the diffusion dynamics. 

The intralayer diffusion is determined by the diffusion constant $D^{[\alpha]}$. Similarly, the interlayer diffusion is dictated by the diffusion constants $D^{[\alpha,\beta]}=D^{[\beta,\alpha]}$. The diffusion across different links of the same layer can be modulated by a weight $w _{ij}^{[\alpha]}$ associated with each undirected link from the replica node $(i,\alpha)$ to the replica node $(j,\alpha)$. Fig. \ref{multi5ddd} presents the main diffusion constant parameters on multiplex networks.  With this notation, the general diffusion equation in a multiplex network is given by~\cite{gomez2013diffusion}, \cite{bianconi2018multilayer}
\begin{equation}
\dot{y}_{i}^{[\alpha]} =  D^{[\alpha]}\sum _{j=1}^{N}w _{ij}^{[\alpha]} \left( y_{j}^{[\alpha]} - y_{i}^{[\alpha]} \right) +\sum _{\beta =1}^{M} D^{[\alpha,\beta]} \left( y_{i}^{[\beta]} - y_{i}^{[\alpha]} \right), 
\label{diffusion1}
\end{equation}
where the first term is the intralayer diffusion and the second term is the interlayer diffusion. Equation (\ref{diffusion1}) can be written as a general diffusion equation in an $(N, M)$ dimensional space with $y \in N\cdot M$ as a vector encoding the dynamical state of all replica nodes of the multiplex network as follows 

\begin{equation}\label{diffMP}
\dot{y}=-\mathcal{L}y,
\end{equation}
where $\mathcal{L}$ is the supra-Laplacian matrix with dimensions $(N\cdot M)\times (N\cdot M)$ with $M$ layers defined as 
 
\begin{equation}\label{supraL}
\mathcal{L} = \Lambda + \Delta
\end{equation}
where matrices $\Lambda$ and $\Delta$ in \eqref{supraL} are given by

\begin{equation}
\Lambda = \left(\begin{IEEEeqnarraybox*}[][c]{,c/c/c/c,}
D^{[1]}L^{[1]} & 0 & \cdots & 0\\
0 & D^{[2]}L^{[2]} & \cdots & 0\\
\vdots & \vdots & \ddots & \vdots \\ 
0 & 0 & \cdots & D^{[M]}L^{[M]}%
\end{IEEEeqnarraybox*}\right)
\label{lambda1}
\end{equation}
and

\begin{equation}\label{delta1}
\begin{tiny}
\Delta = \left(\begin{IEEEeqnarraybox*}[][c]{,c/c/c/c,}
\sum _{\beta \neq 1}D^{[1,\beta]}I_{N\times N} & -D^{[1,2]}I_{N\times N} & \cdots & -D^{[1,M]}I_{N\times N}\\
-D^{[2,1]}I_{N\times N} & \sum _{\beta \neq 2}D^{[2,\beta]}I_{N\times N} & \cdots & -D^{[2,M]}I_{N\times N}\\
\vdots & \vdots & \ddots & \vdots \\ 
-D^{[M,1]}I_{N\times N} & -D^{[M,2]}I_{N\times N} & \cdots & \sum _{\beta \neq M}D^{[M,\beta]}I_{N\times N}%
\end{IEEEeqnarraybox*}\right) ,
\end{tiny}
\end{equation}
respectively, where $I_{N\times N}$ is the $N\times N$ identity matrix. In matrix (\ref{lambda1}) $L^{[\alpha]}$ indicates the Laplacian matrix in each layer whose entries are $L _{ij}^{[\alpha]} = s _{i}^{[\alpha]}\delta _{ij}-w _{ij}^{[\alpha]}$,
being $s _{i}^{[\alpha]}$ the strength of the replica node $(i,\alpha)$ defined by $s _{i}^{[\alpha]} = \sum _{j}w _{ij}^{[\alpha]}$, and the dynamical vector $y$ can be written as $y = \left(y^{[1]} \; y^{[2]}\; \ldots  y^{[M]}\right)^{\top}$, where the $y^{[\alpha]}$ indicates the $N$ column vector of elements $y^{[\alpha]}_{i}$ with $i=1,2,...,N$. In an undirected multiplex network with diffusion constant $D^{[\alpha ,\beta ]}=D^{[\beta ,\alpha ]}$, we assume that the multiplex network, including its interlinks, is connected since the supra-Laplacian matrix $\mathcal{L}$ is symmetric and semi-positive definite with real eigenvalues $\Lambda _{n}>0$~\cite{gomez2013diffusion}. We also assume that it is possible to reach any other node by following a combination of intralinks and interlinks from a replica node.

\section{Problem Statement}

We consider the following distributed optimization problem subject to a multilayer network interaction, which might be interpreted as an interdependent infrastructure, as it is shown in Fig. \ref{multi5ddd}. 

\begin{equation}
\min _{x \in \mathbb{R}}  f(x)= \sum _{\alpha = 1}^{M} \sum _{i=1}^{N}f_{i}^{[\alpha]}(x),
\label{opproblem1}
\end{equation}
where $f_{i}^{[\alpha]}:\mathbb{R} \rightarrow \mathbb{R}$ is assumed to be Lipschitz differentiable convex cost function exclusive to agent $(i,\alpha)$ in layer $\alpha$. It is important to note that the cost of each agent is dependent on the global variable $x$ across all layers. Note also that we have assumed that the decision variable $x$ is in $\mathbb{R}$, but it can be extended to $\mathbb{R}^m$ using the definition of the Kronecker product similar to single layer distributed optimization \eqref{ccm001}. The interaction among replica nodes and between layers of networks can be a cyberphysical multiplex network. To guarantee there exists a unique optimal solution $x^*$ to Problem \eqref{opproblem1}, we assume the following conditions on the gradients.

\begin{assumption}
      The gradients of the multilayer cost functions $f_i^{[\alpha]}$ are Lipschitz continuous in every layer and inter layers $\alpha$
    \begin{equation*}
    ||\nabla f_i^{[\alpha]}(x)-\nabla f_i^{[\alpha]}(x')||_2 \leq L||x-x'||_2, \quad \forall i \in V, \forall \alpha \in \mathcal{M}.
    \end{equation*}
\end{assumption}

In the distributed multilayer optimization problem (\ref{opproblem1}), we should consider the heterogeneous relationship between nodes. In each layer, the weighted interaction might be different, i.e., $w_{ij}^{[\alpha]}\neq w_{ij}^{[\beta]}$. The heterogeneity of each layer can also be included in the interaction model using the intralayer diffusion constant $D^{[\alpha]}$. Also, the interlayer diffusion constants allow us to include several possible different interactions between layers; for instance, the constant $D^{[\alpha, \beta]}$ can be understood as an energy conversion factor between different energy networks (See Fig. \ref{multi5ddd} for a description of the interaction model of multiplex networks and its dynamics). The constraints governing the diffusion within each network ($w_{ij}^{[\alpha]}$, $D^{[\alpha]}$) and between networks ($D^{[\alpha, \beta]}$) are included into the supra-Laplacian \eqref{supraL}, extending the traditional Laplacian to incorporate more challenging higher-order interactions between nodes and layers. First, we assign each agent an estimation $y_{i}^{[\alpha]} \in \mathbb{R}$ for the variable $x \in \mathbb{R}$ to develop a distributed solution to address the optimization problem (\ref{opproblem1}). Due to this multiplex network interaction, we propose to solve the following equivalent distributed optimization problem

\begin{eqnarray}\label{opproblem2}
&&\min _{y\in \mathbb{R}^{N\cdot M}}\tilde{f}(y)= \sum _{\alpha = 1}^{M} \sum _{i=1}^{N}f_{i}^{[\alpha]}(y_{i}^{[\alpha]}) \nonumber \\
&&\text{s.t.} \nonumber \\
&& \quad D^{[\alpha]}\sum_{j=1}^{N}w_{ij}^{[\alpha]}\left( y_{i}^{[\alpha]} - y_{j}^{[\alpha]} \right)=\sum _{\beta =1}^{M} D^{[\alpha,\beta]} \left( y_{i}^{[\beta]} - y_{i}^{[\alpha]} \right), \nonumber\\
&& \qquad  \forall \; \alpha \in \mathcal{M}.
\end{eqnarray}

To illustrate the challenges of the multiplex network optimal consensus, we observe that Problem \eqref{opproblem2} should satisfy the Karush-Kunh Tucker (KKT) sufficient and necessary conditions for the global optimal solution $y_i^{[\alpha]*}=x^*$ for all $\alpha \in \mathcal{M}$ and $i\in V$ as follows

\begin{equation*}
    \sum _{\alpha = 1}^{M} \sum _{i=1}^{N}\nabla_{y_i^{[\alpha]}}f_{i}^{[\alpha]}(y_{i}^{[\alpha]})=0
\end{equation*}
and for all $\alpha \in \mathcal{M}$
\begin{equation}\label{consmulti}
    D^{[\alpha]}\sum_{j=1}^{N}w_{ij}^{[\alpha]}\left( y_{i}^{[\alpha]} - y_{j}^{[\alpha]} \right)=\sum_{\beta =1}^{M} D^{[\alpha,\beta]} \left( y_{i}^{[\beta]} - y_{i}^{[\alpha]} \right).
\end{equation}

It is observed that \eqref{consmulti} should be satisfied to guarantee the convergence to a global solution. From \eqref{consmulti}, we can obtain the conditions for multiplex consensus as follows. First, for the intralayer consensus

\begin{equation}
    y_i^{[\alpha]}=\frac{1}{D^{[\alpha]}\sum_{j=1}^{N}w_{ij}}\sum_{j=1}^{N}w_{ij}^{[\alpha]} y_{j}^{[\alpha]}, 
\end{equation}
and for the interlayer consensus, we have

\begin{equation}
    y_i^{[\alpha]}=\frac{1}{\sum_{\beta =1}^{M} D^{[\alpha,\beta]}}\sum_{\beta =1}^{M} D^{[\alpha,\beta]} y_{i}^{[\beta]}.
\end{equation}

Furthermore, notice that Problem \eqref{opproblem2} is equivalent to the intralayer constraint condition

\begin{equation}
    y_i^{[\alpha]}=\frac{1}{D^{[\alpha]}\sum_{j=1}^{N}w_{ij}}\left(\sum_{j=1}^{N}w_{ij}^{[\alpha]} y_{j}^{[\alpha]}- \sum _{\alpha = 1}^{M} \sum _{i=1}^{N}\nabla_{y_i^{[\alpha]}}f_{i}^{[\alpha]}(y_{i}^{[\alpha]})\right),
\end{equation}
and the interlayer constraint condition

\begin{equation}
    y_i^{[\alpha]}=\frac{1}{\sum_{\beta =1}^{M} D^{[\alpha,\beta]}}\left(\sum_{\beta =1}^{M} D^{[\alpha,\beta]} y_{i}^{[\beta]} - \sum _{\alpha = 1}^{M} \sum _{i=1}^{N}\nabla_{y_i^{[\alpha]}}f_{i}^{[\alpha]}(y_{i}^{[\alpha]})\right).
\end{equation}

Therefore, every agent that adopts the following dynamic satisfying simultaneously the intra- and interlayer constraints

\begin{eqnarray}
    \dot{y}_{i}^{[\alpha]}=\left(\sum_{j=1}^{N}w_{ij}^{[\alpha]} y_{j}^{[\alpha]}- \sum _{\alpha = 1}^{M} \sum _{i=1}^{N}\nabla_{y_i^{[\alpha]}}f_{i}^{[\alpha]}(y_{i}^{[\alpha]})\right) \nonumber\\
    +\left(\sum_{\beta =1}^{M} D^{[\alpha,\beta]} y_{i}^{[\beta]} - \sum _{\alpha = 1}^{M} \sum _{i=1}^{N}\nabla_{y_i^{[\alpha]}}f_{i}^{[\alpha]}(y_{i}^{[\alpha]})\right) \nonumber\\
    -\left(D^{[\alpha]}\sum_{j=1}^{N}w_{ij}y_i^{[\alpha]}+\sum_{\beta =1}^{M} D^{[\alpha,\beta]}y_i^{[\alpha]}\right)
\end{eqnarray}
will converge to the optimal solution $y^*$, with $y_i^{[\alpha]}=y_j^{[\alpha]}=y^*$ for all $i,j \in V$ and $y_i^{\alpha}=y_i^{\beta}=y^*$ for all $\alpha, \beta \in \mathcal{M}$, and since this solution satisfies the KKT conditions for Problem \ref{opproblem1} then $x^*=y^*$. Notice that this dynamic for each agent is not distributed, so we need to propose a distributed dynamic for each agent able to estimate the global gradient $\sum _{\alpha = 1}^{M} \sum _{i=1}^{N}\nabla_{y_i^{[\alpha]}}f_{i}^{[\alpha]}(y_{i}^{[\alpha]})$. It has been observed that the distributed algorithm $y=-\nabla \tilde{f}(y)-Ly$ cannot converge to the optimal solution since local gradients are generally different \cite{kia2015dynamic}. Considering this observation, in the next section, we propose the concept of the augmented supra-Langragian for including the implicit gradient tracking for the intralayer and interlayer gradients.   

First, we proceed to solve the following equivalent supra-Laplacian-based consensus-constrained problem

\begin{eqnarray}\label{opproblem3}
&&\min _{y\in \mathbb{R}^{N\cdot M}}\tilde{f}(y)= \sum _{\alpha = 1}^{M} \sum _{i=1}^{N}f_{i}^{[\alpha]}(y_{i}^{[\alpha]}) \nonumber \\
&&\text{s.t. }\quad \mathcal{L}y =  0 _{N\cdot M}.
\end{eqnarray}
where the supra-Laplacian constraint $0 _{N\cdot M}=\mathcal{L}y$ is a multiplayer consensus constraint that assures that the local estimate for each agent converges to the same point. 

In the next sections, we propose two saddle-point dynamics algorithms to solve the distributed optimization problem with supra-Laplacian-based constraints. 

\section{Distributed Saddle-Point Dynamics for Multiplex Networks}
In this section, we introduce the distributed continuous-time optimization algorithm for multiplex networks based on saddle-point dynamics of the supra-Lagrangian with supra-Laplacian constraints. Considering the constrained structure of the multiplex network optimization problem \eqref{opproblem3}, we extend the saddle-point dynamics for continuous-time single-layer optimization to the multiplex case using its distributed implementation.

For Problem \eqref{opproblem3}, we propose the following augmented supra-Lagrangian function $\mathcal{L}_{a}: \mathbb{R}^{N\cdot M}\times \mathbb{R}^{N\cdot M}\rightarrow \mathbb{R}$ 
\begin{equation}
\mathcal{L}_{a}(y,\lambda) = \tilde{f}(y)+\lambda ^{T}\mathcal{L}y +\frac{1}{2}y^{T}\mathcal{L}y,
\label{augmented1}
\end{equation}
we can obtain the optimal solution of \eqref{opproblem3} by solving its corresponding saddle-point problem 
\begin{equation*}
    \min_y\max_{\lambda}\mathcal{L}_a(y,\lambda),
\end{equation*}
and the associated saddle-point dynamics yields
\begin{equation}
\dot{y}=-\nabla_{y}\mathcal{L}_{a}(y,\lambda)=-\nabla_{y} \tilde{f}(y)-\mathcal{L}y-\mathcal{L}\lambda
\label{saddleflow1}
\end{equation}
\begin{equation}
\dot{\lambda}=\nabla_{\lambda}\mathcal{L}_{a}(y,\lambda)= \mathcal{L}y,
\label{saddleflow2}
\end{equation}
which leads to the following dynamics for each agent $i$ on each layer $\alpha$
\begin{eqnarray}
    \dot{y}_{i}^{[\alpha]} = -\nabla_{y_i^{[\alpha]}}f_i^{[\alpha]} +D^{[\alpha]}\sum _{j=1}^{N}w _{ij}^{[\alpha]} \left( y_{j}^{[\alpha]} - y_{i}^{[\alpha]} \right) \nonumber\\
    +\sum _{\beta =1}^{M} D^{[\alpha,\beta]} \left( y_{i}^{[\beta]} - y_{i}^{[\alpha]} \right)\nonumber\\
    +D^{[\alpha]}\sum _{j=1}^{N}w _{ij}^{[\alpha]} \left( \lambda_{j}^{[\alpha]} - \lambda_{i}^{[\alpha]} \right) \nonumber\\
    +\sum _{\beta =1}^{M} D^{[\alpha,\beta]} \left( \lambda_{i}^{[\beta]} - \lambda_{i}^{[\alpha]} \right)
\label{saddledyn}
\end{eqnarray}

\begin{remark}
    An important feature of the proposed saddle-point dynamics \eqref{saddledyn} is that the augmented term $\frac{1}{2}y^{T}\mathcal{L}y$ of $\mathcal{L}_a$ allows the convergence without the strict convexity of $f_i^{\alpha}$. As it has been mentioned and can be observed in Fig. \ref{control31}, the augmented term introduces a derivative feedback term in the saddle-point dynamics that relaxes the strict convexity property to only convexity.   
\end{remark}

Fig. \ref{control31} illustrates a block-diagram representation of the saddle-point flow (\ref{saddleflow1}) - (\ref{saddleflow2}). The first three upper blocks refer to the primal gradient dynamics. The following two, are taking the multiplex damping via augmentation. The final two bottom blocks are the dual integrator dynamics. The supra-Laplacian blocks, in the input and output of the dual integrator dynamics, are the skew-symmetric multiplex coupling.

The following lemma describes the optimality conditions for the supra-Laplacian-based Lagrangian (\ref{augmented1}) and the equilibrium of the multiplex saddle-point dynamics (\ref{saddleflow1})-(\ref{saddleflow2}).

\begin{lemma}\label{lem1}
Suppose we have a symmetric supra-Laplacian $\mathcal{L}=\mathcal{L}^{T}\in \mathbb{R}^{N\cdot M \times N\cdot M}$ associated with a weighted, undirected, and connected multiplex network. Consider convexity of the function $f_{i}^{[\alpha]}$ for all $(i,\alpha)$, with $i\in V=\{ 1,...,N \}$ and $\alpha\in \mathcal{M}=\{ 1,...,M\}$ in supra-Lagrangian (\ref{augmented1}). The following conditions follow:

\begin{itemize}
    \item[1] If we assume that $(y^{\star},\lambda ^{\star})\in \mathbb{R}^{N\cdot M}\times \mathbb{R}^{N\cdot M}$ is a saddle point of (\ref{augmented1}). Then, $(y^{\star},\lambda ^{\star}+\gamma 1_{N\cdot M})$ is also a minimax point for any $\gamma \in \mathbb{R}$.
    \item[2] Suppose $(y^{\star},\lambda ^{\star})\in \mathbb{R}^{N\cdot M}\times \mathbb{R}^{N\cdot M}$ is a saddle point of (\ref{augmented1}). Then, $y^{\star} = x^{\star}1_{N\cdot M}$ is the solution of primal problem (\ref{opproblem1}) with $x^{\star} \in \mathbb{R}$.
    \item[3] There exist saddle-point $(y^{\star},\lambda ^{\star})$ of the multiplex augmented Lagrangian (\ref{augmented1}) satisfying
    \begin{equation*}
\mathcal{L}\lambda ^{\star}+\nabla_y \tilde{f}(y^{\star})=0_{N\cdot M}.
\end{equation*}
\end{itemize}
\end{lemma}

\begin{proof}
It follows from the Karush-Kuhn Tucker conditions applied to the Lagrangian function (\ref{augmented1}) and the definition of a saddle-point in convex optimization \cite{bertsekas2009convex}.   
\end{proof}

Now, it is presented the convergence analysis of the proposed saddle-point dynamics for multiplex networks (\ref{saddleflow1})-(\ref{saddleflow2}). The main result for convergence analysis of the proposed primal-dual saddle-point dynamics is presented in the following theorem.

\begin{figure}[t]
\begin{center}

\includegraphics[width=.4\textwidth]{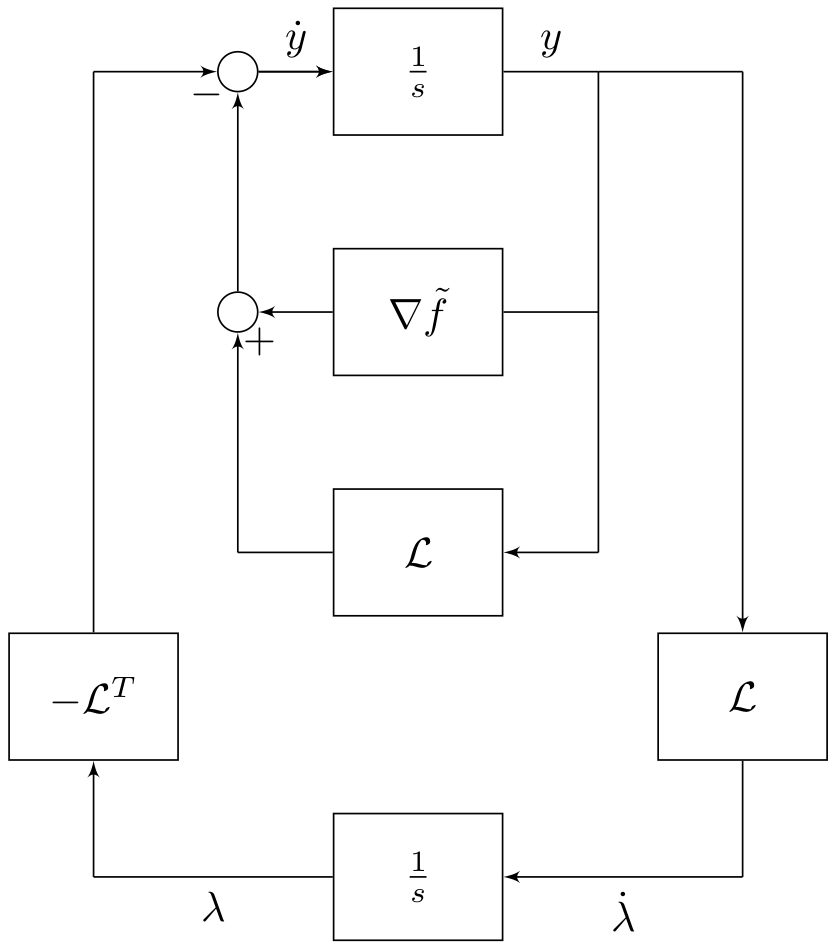}

\end{center}
\caption{Block-diagram representation of the saddle-point dynamics (\ref{saddleflow1}) - (\ref{saddleflow2}) for the supra-Laplacian saddle-point dynamics.}
\label{control31}
\end{figure}

\begin{theorem}\label{thm2}
Suppose we have a twice continuously differentiable convex function $\tilde{f}: \mathbb{R}^{N\cdot M} \rightarrow \mathbb{R}$. The convergence to a unique saddle point $(y^{\star}, \lambda^{\star})$ of every trajectory $(y(t), \lambda (t))$ of the multiplex saddle-point dynamics (\ref{saddleflow1})-(\ref{saddleflow2}) is determined by the following conditions:

\begin{itemize}
    \item[1] $y^{\star} = x^{\star}1_{N\cdot M}$ is an optimizer of the primal optimization problem (\ref{opproblem1}) with $x^{\star} \in \mathbb{R}$.
    \item[2] $\lambda^{\star} = \bar{\lambda}+\text{average}(\lambda _{0})1_{N\cdot M}$, where $\text{average}(\lambda _{0})$ is the average of the initial conditions for the vector $\lambda(0)$, and $\bar{\lambda}\perp 1_{N\cdot M}$ satisfies 
\begin{equation*}
\mathcal{L}\bar{\lambda}+\nabla_y \tilde{f}(y^{\star})=0_{N\cdot M}
\end{equation*}
\end{itemize}
\end{theorem}

\begin{proof}
See Appendix \ref{Proofthm2}.
\end{proof}
\section{Distributed Gradient Descent for Multiplex Networks}   
In this section, as a variation to the saddle-point dynamics \eqref{saddleflow1}-\eqref{saddleflow2} to reduce the number of parameters to design, we propose a distributed descent algorithm for multiplex networks. Consider the saddle-point dynamics \eqref{saddleflow1}-\eqref{saddleflow2} represented in block-diagram in Fig. \ref{control31}, it is possible to reduce the dynamics  obtaining as a result the corresponding distributed gradient descent flow with supra-Laplacian defined as 
\begin{equation}
K_{1}\dot{y}=-\nabla_y \tilde{f}(y)-\rho \mathcal{L}y .
\label{distributedgflow1}
\end{equation}

Based on the soft-penalty method \cite{bertsekas2009convex}, \cite{yang2016multi}, this distributed supra-Laplacian gradient flow (\ref{distributedgflow1}) can be obtained from  the multiplex optimization problem
\begin{equation}
\min _{y\in \mathbb{R}^{N\cdot M}} \sum _{\alpha =1}^{M} \sum _{i=1}^{N} f_{i}^{[\alpha]}(y_{i}^{[\alpha]}) + \frac{1}{2}\rho y^{T}\mathcal{L}y.
\label{opproblem5}
\end{equation}

Notice that the soft-penalty term $\frac{1}{2}\rho y^{T}\mathcal{L}y$ aims to enforce supra-Laplacian consensus-constraint. Although the minimizer of (\ref{opproblem5}) differs from that of the distributed optimization problem (\ref{opproblem3}), the following outcome demonstrates that if the gains are time-varying, the flow (\ref{distributedgflow1}) can achieve convergence to the optimal solution of problem (\ref{opproblem3}).  


Consider the dynamical system with a time-varying positive gain $\varsigma:\mathbb{R}_{\geq 0}\rightarrow \mathbb{R}_{>0}$ 
\begin{equation}
\dot{y}=-\varsigma (t)\nabla_y \tilde{f}(y)-\mathcal{L}y.
\label{distgflow2}
\end{equation}

 Recall that $\varsigma(t)$ satisfies the persistence condition

\begin{equation}
\quad \int _{0}^{\infty}\varsigma (\tau)d\tau = \infty ,\quad \text{and}\, \lim _{t\rightarrow \infty}\varsigma (t) = 0.
\label{persistencec1}
\end{equation} 

In other words, the supra-Laplacian constraint in the distributed gradient flow (\ref{distgflow2}) becomes more influential than the gradient descent as time passes due to a decaying but non-integrable gain $\varsigma(t)$. The convergence of the distributed gradient flow for multiplex networks is given by the following theorem.

\begin{theorem}\label{thm3}
If $\tilde{f}:\mathbb{R}^{N\cdot M}\rightarrow \mathbb{R}$ is a radially unbounded, twice continuously differentiable convex function and $\varsigma: \mathbb{R}{\geq 0}\rightarrow \mathbb{R}{>0}$ satisfies the condition of persistence (\ref{persistencec1}), then the multiplex gradient dynamics (\ref{distgflow2}) converges to the single optimal solution $y^{\star}=x^{\star}1_{N\cdot M}$ for all $t \geq 0$.
\end{theorem}
\begin{proof}
See Appendix \ref{Proofthm3}.
\end{proof}

Recently, some results on continuous-time multiagent neurodynamic systems to tackle nonsmooth distributed optimization problems with general local convex constraints but not necessarily differentiable functions have been proposed \cite{ma2019novel}\cite{guo2024neurodynamic}. These results have a similar simple structure for single-layer networks as the one proposed in this work for multilayer networks. It would be interesting to use the neurodynamic approach to extend the proposed algorithm to include nonsmooth functions.

In the next section, we illustrate the effectiveness of the proposed algorithms in several numerical experiments. 

\section{Numerical Experiments}

\subsection{Multiplex Network with Two Layers and Three Nodes in Each Layer}

We consider a multiplex network with $M=2$ layers and $N=3$ nodes in each layer.
We study the distributed optimization problem via the saddle-point algorithm to solve the following optimization problem
\begin{eqnarray}
&&\min _{y\in \mathbb{R}^{3\cdot 2}}\tilde{f}(y)=\sum _{i=1}^{3}f^{[1]}(y_{i}^{[1]})+\sum _{i=1}^{3}f^{[2]}(y_{i}^{[2]})\nonumber \\
&&\text{s.t. }\quad \mathcal{L}y =  0 _{3\cdot 2},
\label{opproblem3b}
\end{eqnarray}
where $f^{[1]}(y_{i}^{[1]})=\frac{1}{2} (y_{i}^{[1]})^{2}+iy_{i}^{[1]}$, $f^{[2]}(y_{i}^{[2]})=\frac{1}{2} (y_{i}^{[2]})^{2}+(i+3)y_{i}^{[2]}$, and the supra-Laplacian matrix with $D^{[1]}=D^{[2]}=D^{[1,2]}=1$. The results of the optimal consensus are shown in Fig. \ref{multi3}. We can observe a rapid convergence to the optimal $y^{\star}$ due to the value of the interdiffusion constant $D^{[1,2]}$ and the number of nodes in each layer.



\begin{figure}[!h]
\includegraphics[width=3.5in]{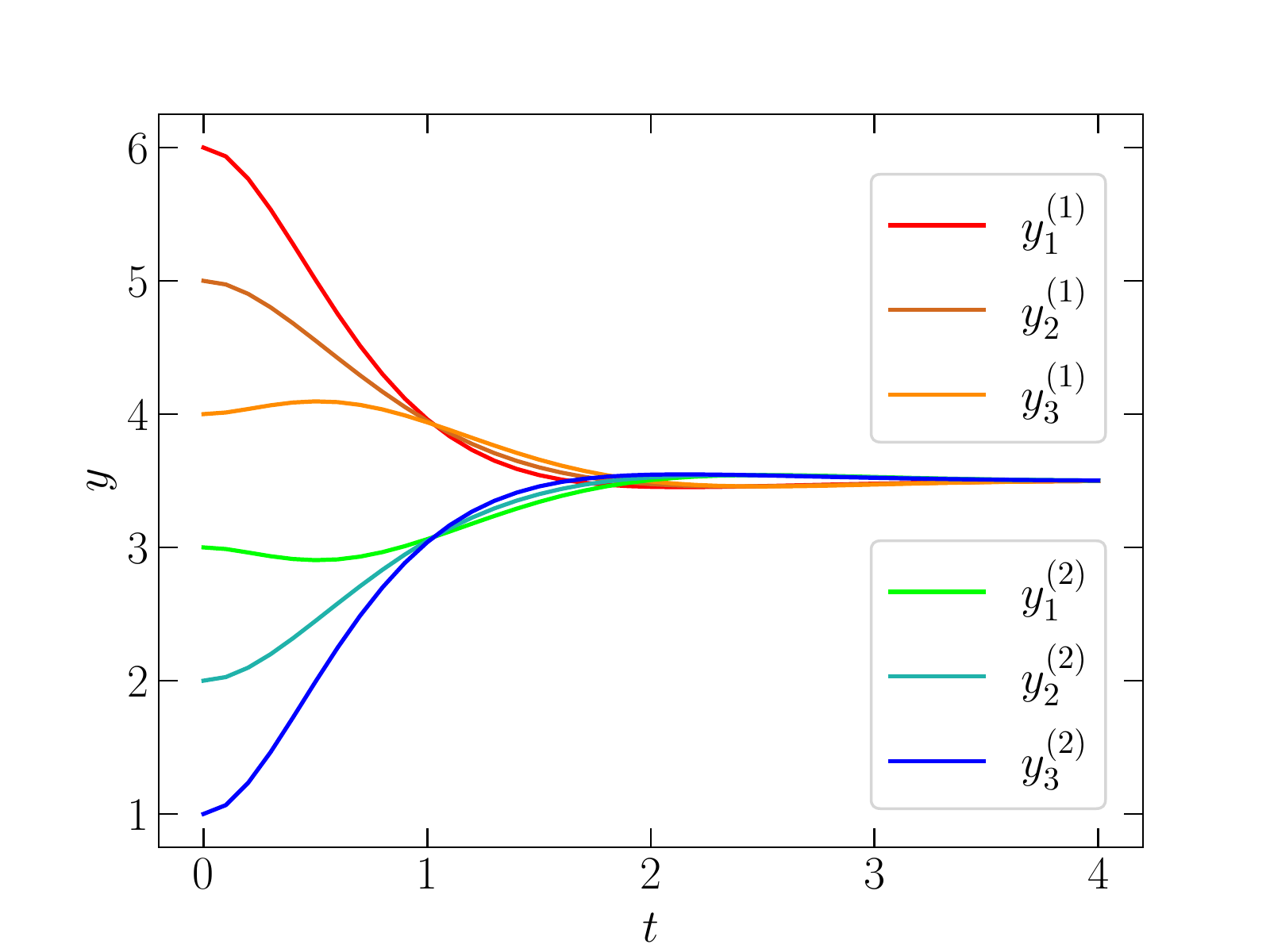}
\caption{Simulation results for the distributed optimization problem (\ref{opproblem3b}) with supra-Laplacian matrix.}
\label{multi3}
\end{figure}

Now, to the same problem, we apply the distributed gradient descent considering the time-varying flow (\ref{distgflow2}) and a time-varying positive gain $\varsigma (t,\theta)$ defined by
\begin{equation*}
\varsigma (t,\theta)=\frac{1}{\theta + t}.
\end{equation*}

The results of the time-varying flow (\ref{distgflow2}) associated with the optimization problem (\ref{opproblem3b})  are depicted in Fig. \ref{grad3_unidas}. Simulations are showing remarkable results, as $\theta \gg 1$ the optimal value is getting closer to the one obtained by the distributed primal-dual saddle-point algorithm; even, with a large $\theta$ the consensus times are also getting similar.

\begin{figure*}[htb]
\centering
  \begin{tabular}{@{}ccc@{}}
    \includegraphics[width=.32\textwidth]{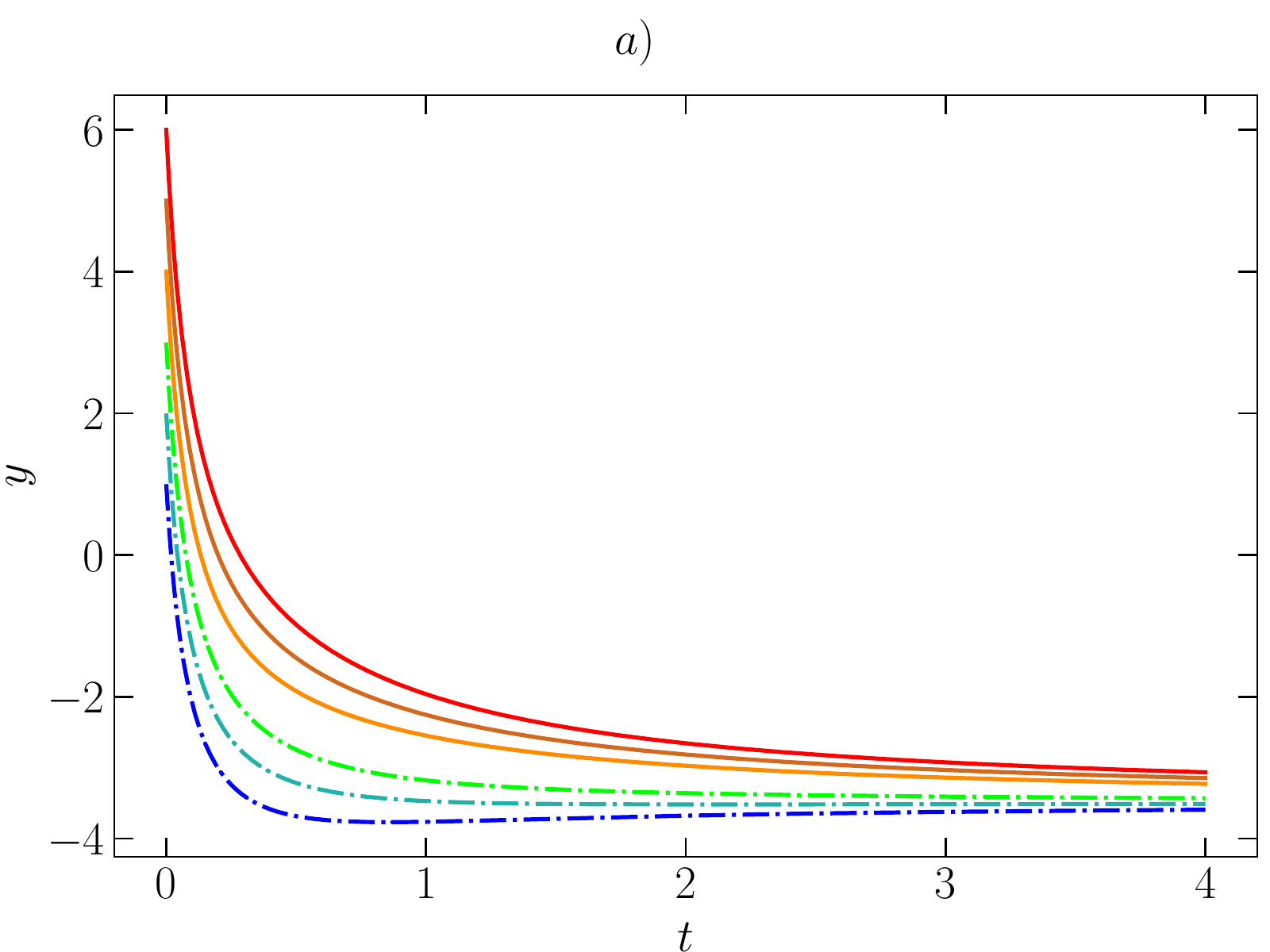} &
    \includegraphics[width=.32\textwidth]{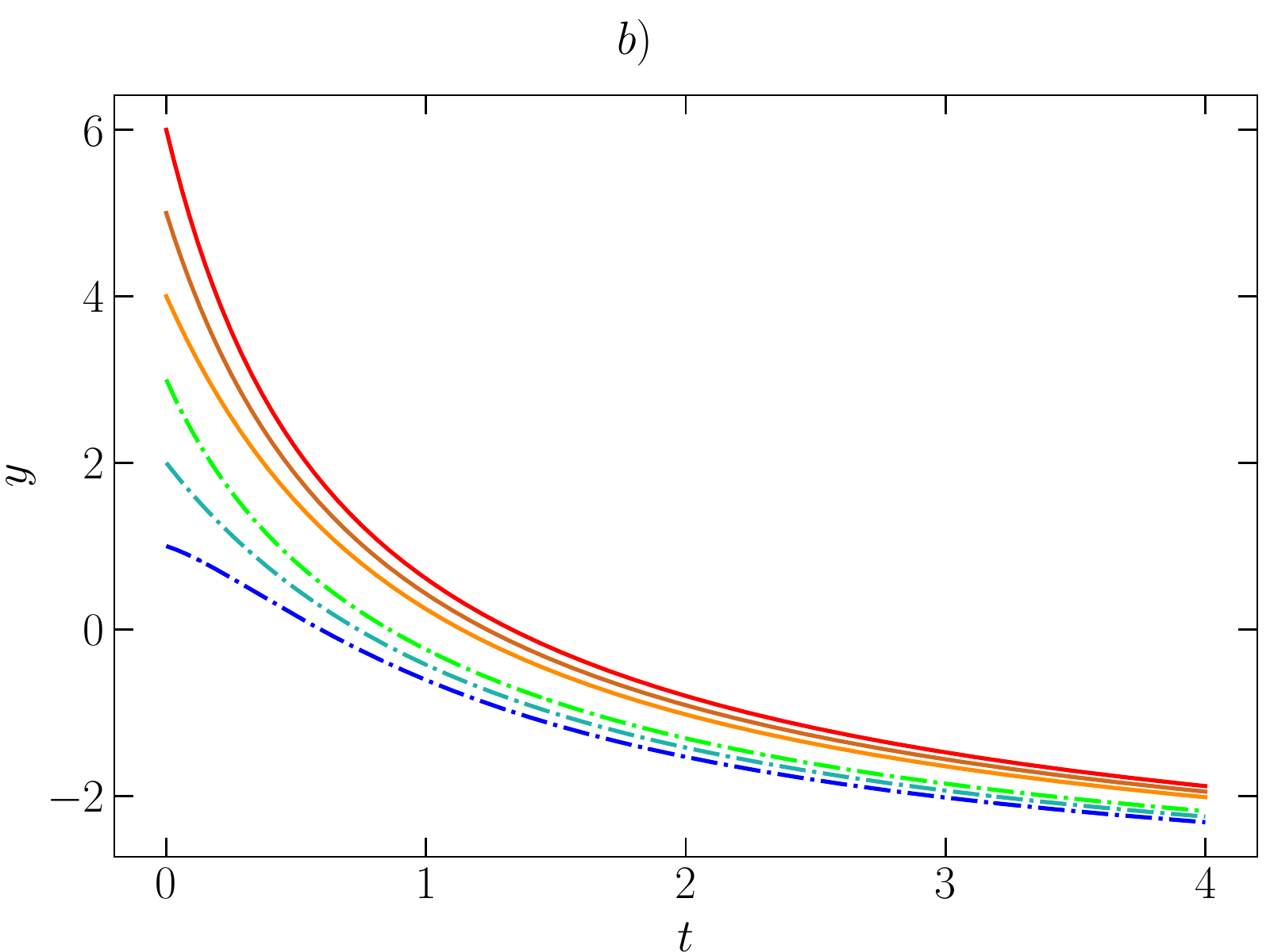} &
    \includegraphics[width=.32\textwidth]{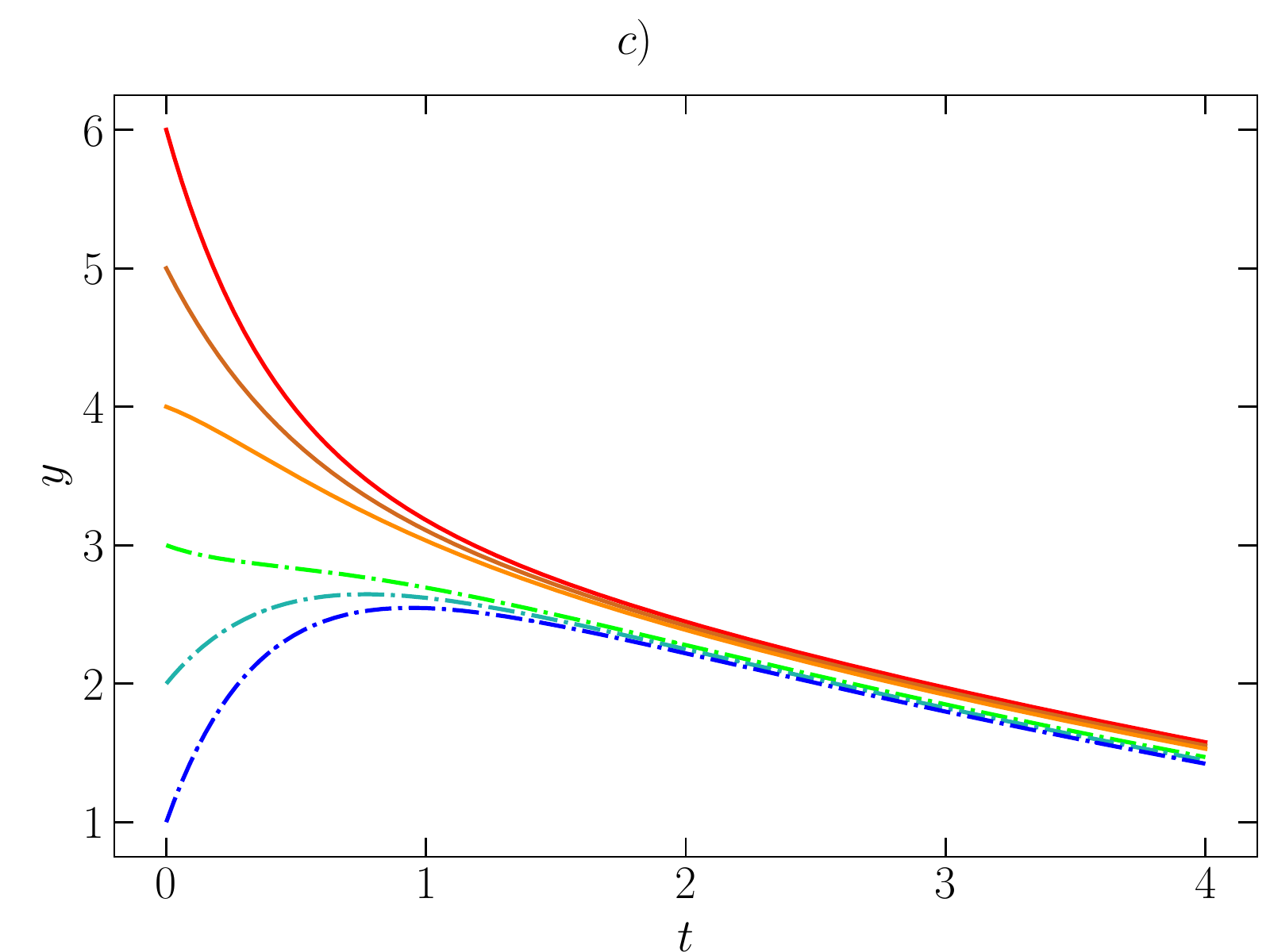} \\
    \includegraphics[width=.32\textwidth]{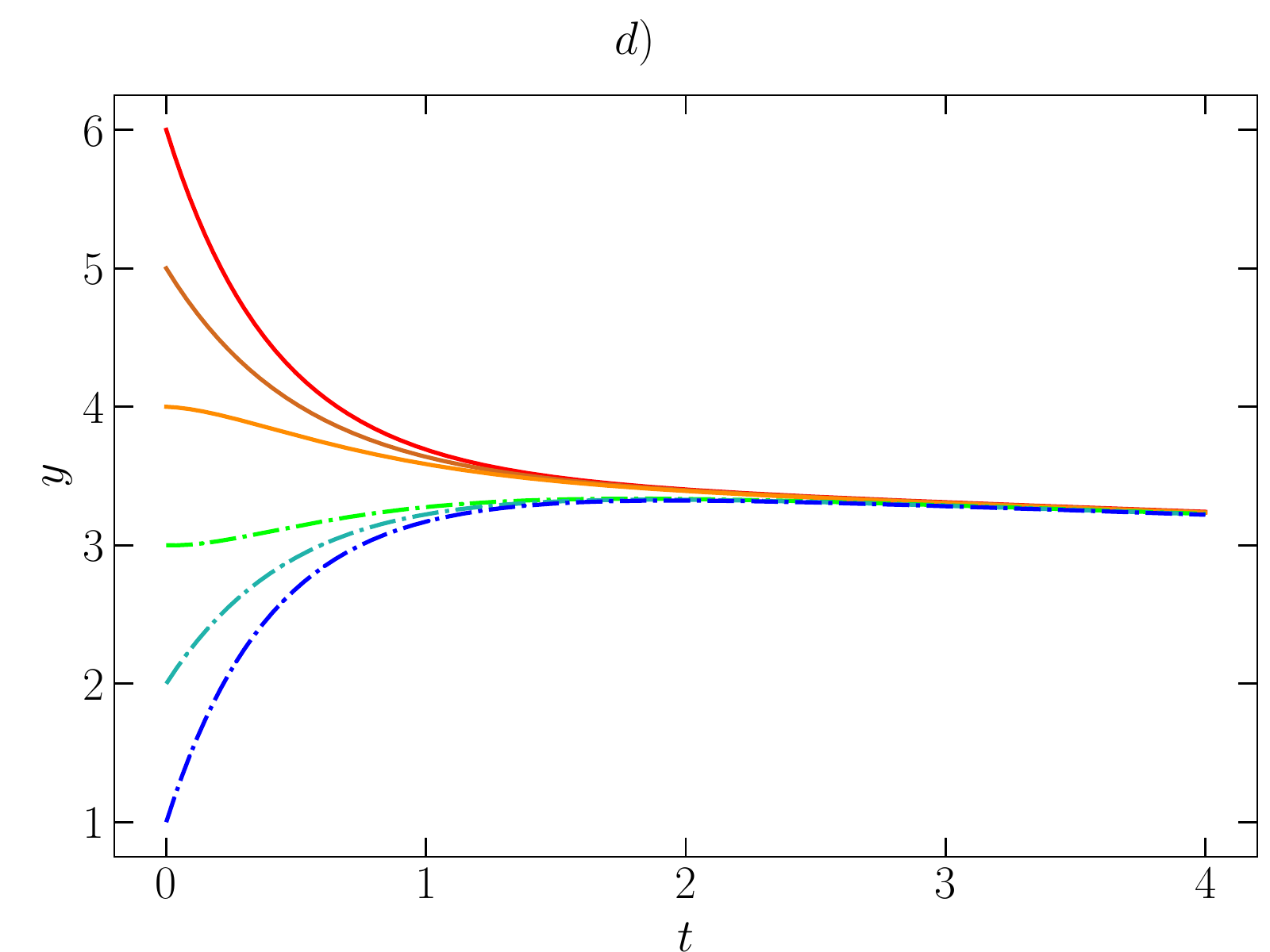} &
    \includegraphics[width=.32\textwidth]{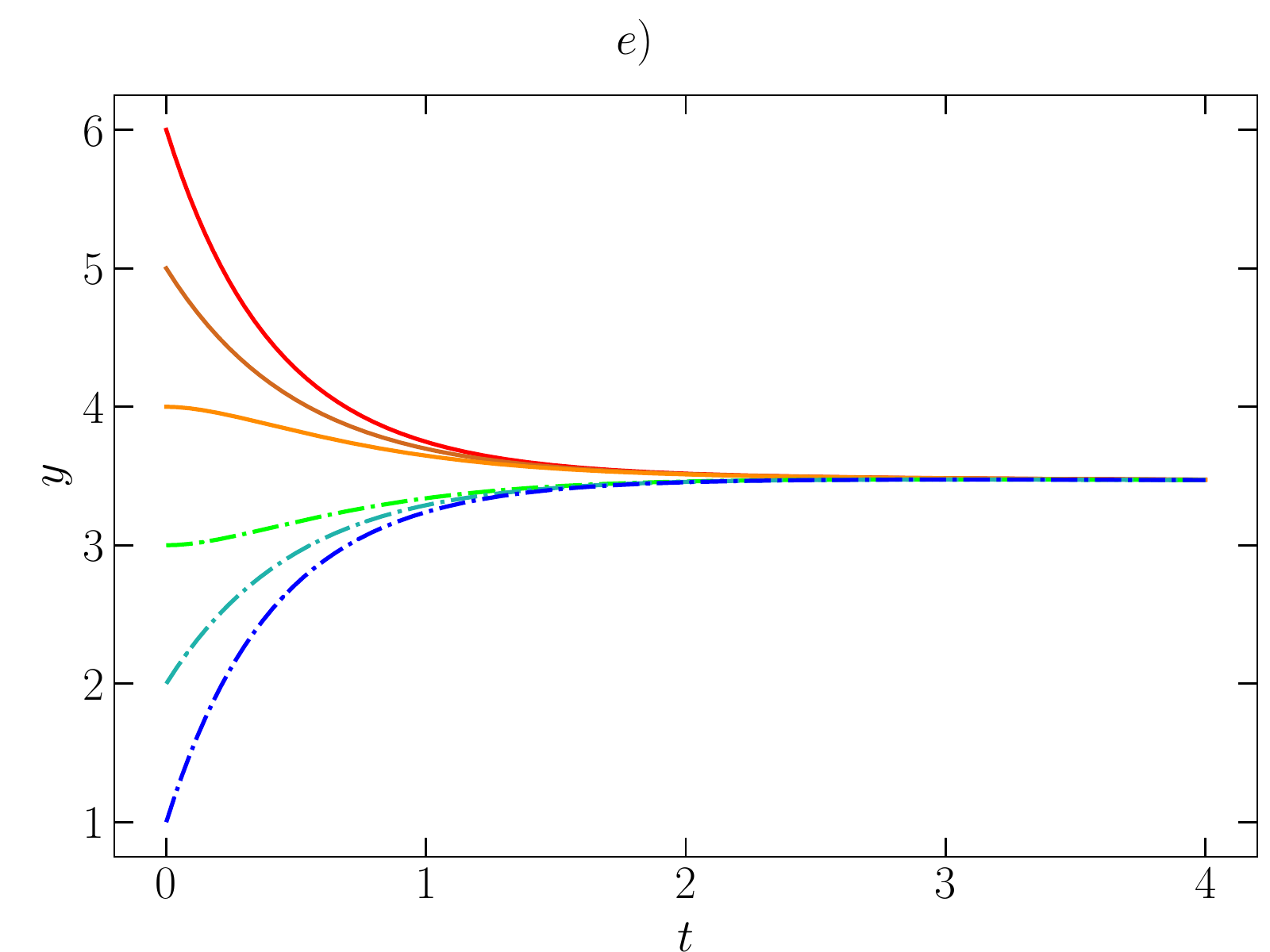} &
    \includegraphics[width=.32\textwidth]{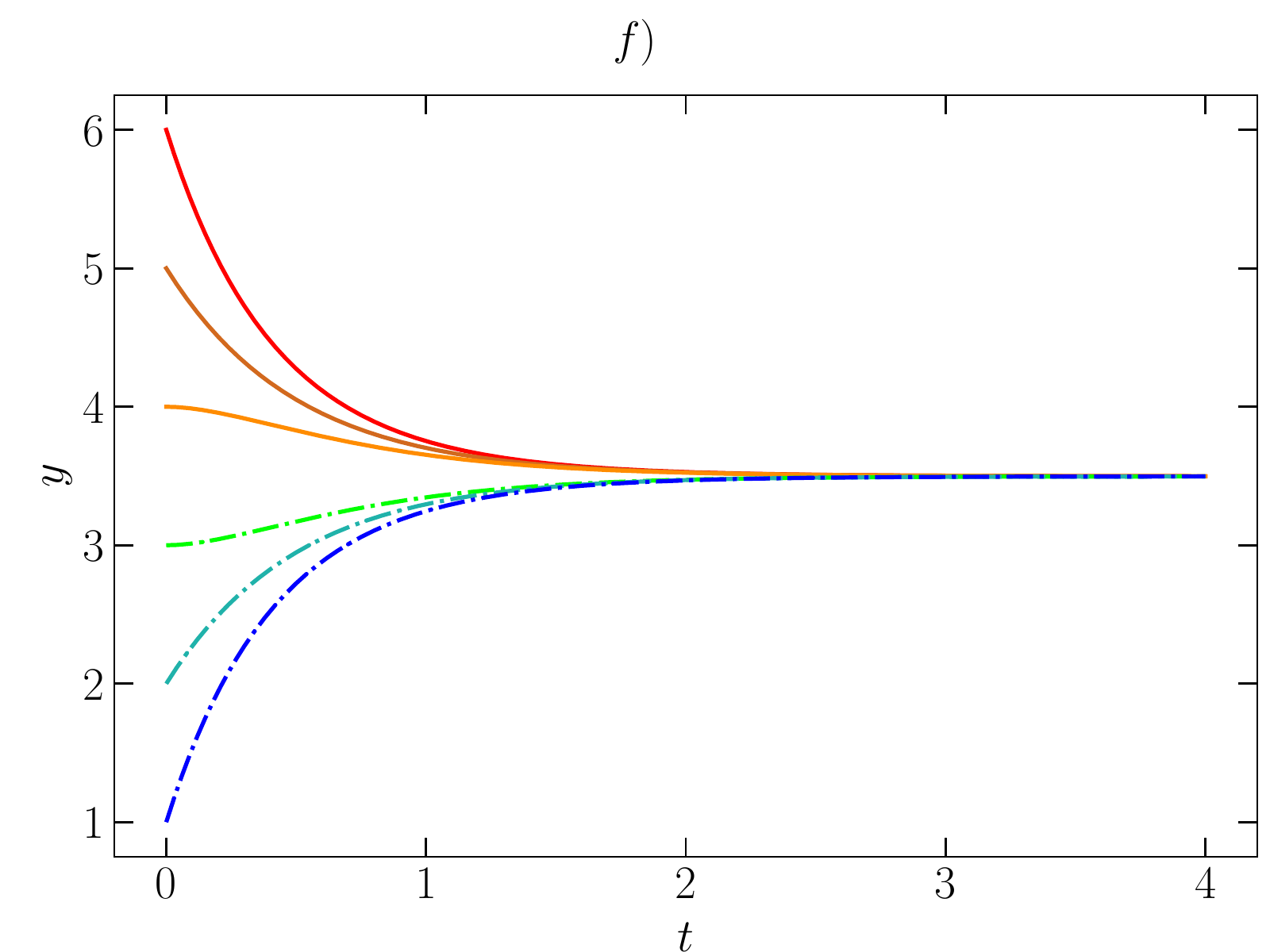} \\
  \end{tabular}
  \caption{Numerical results for the distributed optimization problem (\ref{opproblem3b}) resolved by distributed gradient descent considering the time-varying flow (\ref{distgflow2}) and a time-varying positive gain: a) $\varsigma(t,0.1)$, b) $\varsigma(t,1)$, c) $\varsigma(t,10)$, d) $\varsigma(t,100)$, e) $\varsigma(t,1000)$, and f) $\varsigma(t,10000)$. The conventions are the same as presented in Fig. \ref{multi3}.}
  \label{grad3_unidas}
\end{figure*} 


\subsection{Consensus Convergence to the Optimal as a Function of the Interlayer Diffusion Constant}  

The previous results, together with the fact that the interlayer diffusion constant $D_{x}$ is driving a critical phenomenon, motivate us to study the behavior of the consensus convergence to the optimal $y^{\star}$ times $t_{c}$ as a function of the interlayer diffusion constant $D_{x}$, considering the supra-Laplacian second eigenvalue connected with diffusion times. In Fig. \ref{consensusg1},  we depict the time at which each $y_{i}^{[\alpha]}$ does reach the optimal for a two-layer multiplex network, wherein in the first layer we have a ring topology, and in the second one we have a full-connected network. We can observe that, as it was expected: First, the nodes will reach the optimal faster as $D_{x}\rightarrow \infty$. Second, we have found values of $D_{x}$ where there is present a discontinuity. That is to say, we have encountered a situation that is validating our hypothesis about the possible presence of critical phenomena driven by $D_{x}$ and its relationship with the diffusion dynamics, which at the same time, is related to the consensus and optimization processes in the multiplexes. As we can observe in Fig. \ref{consensusg1}, there are points where there are sudden changes that do not obey certain continuity but are presenting an optimization time process.

\begin{figure}[!h]
\includegraphics[width=3.4in]{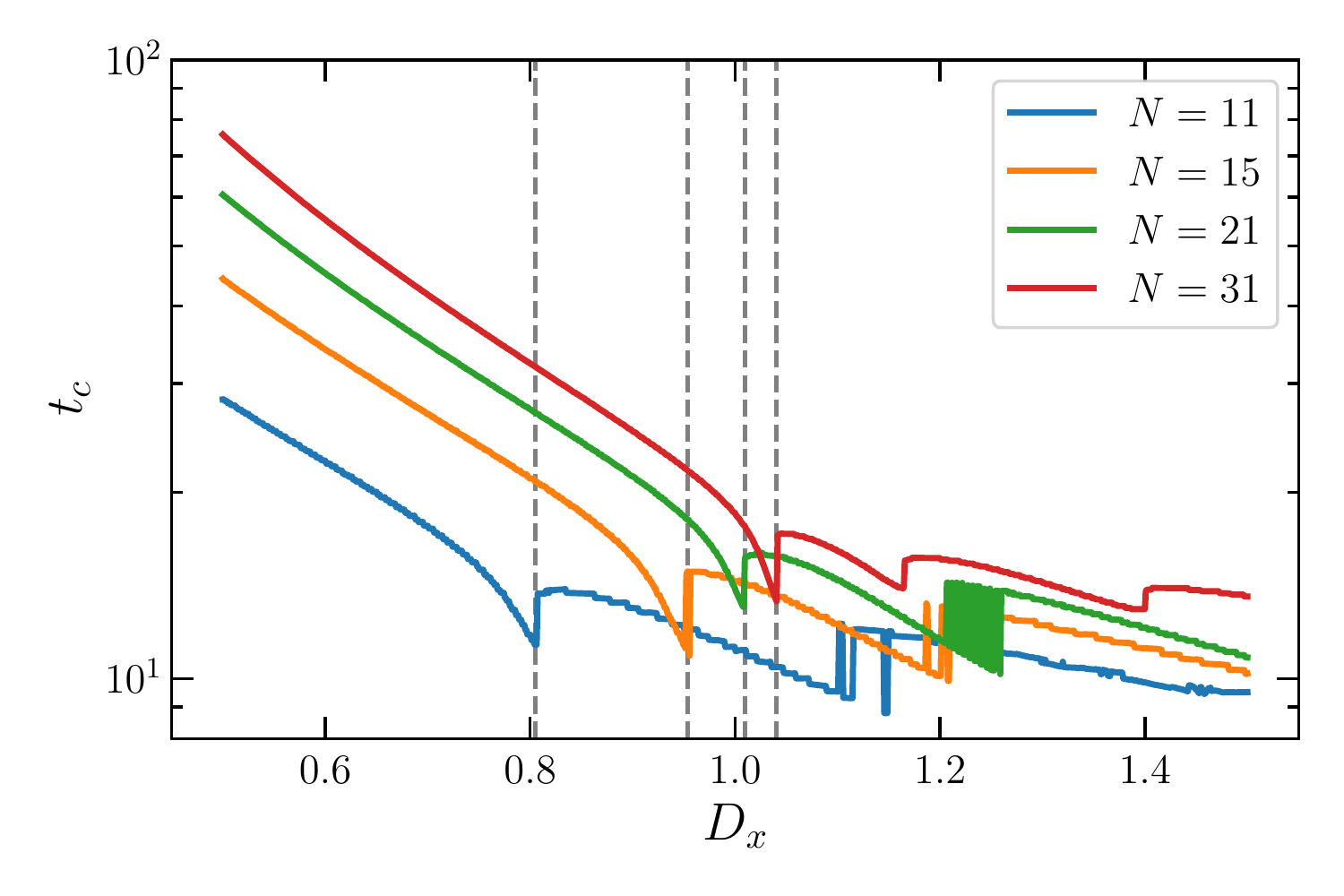}
\caption{Consensus convergence to the optimal as a function of the interlayer diffusion constant.}
\label{consensusg1}
\end{figure}

\subsection{Coordinated Dispatch for Multienergy System}

As a final example, let us explore an \emph{applied} situation where 2 + 1 restrictions are present. We study the case of three layers of micro-grids of seven power generators (each one) which are connected between them. Such is the case of coordinate dispatch for interdependent infrastructure networks. The multilayer dispatch problem seeks to minimize the total cost of operating both power and gas networks. The global objective function and power and gas constraints are defined as follows
\begin{equation}
\min _{x \in \mathbb{R}}\sum _{\alpha = 1}^{3}\sum _{i = 1}^{7}f_{i}^{[\alpha]}(x) ,
\end{equation}
which translated to primal-dual saddle-point language, gives us the problem of distributed power optimization
\begin{equation*}
\min_{q} \left( \sum_{\tau \in \mathcal{T}}C_{\tau}^{[\mathcal{T}]}(p_{\tau}) + \sum_{k \in \mathcal{K}}C_{k}^{[\mathcal{K}]}(p_k) + \sum_{j \in \mathcal{G}}C_{j}^{[\mathcal{G}]}(g_j)\right) ,
\end{equation*}
subject to
\begin{equation*}
\mathcal{L}q = 0,
\end{equation*}
\begin{equation*}
\sum_{\tau \in \mathcal{T}}p_{\tau} + \sum_{k \in \mathcal{K}}p_k = P_{D},
\end{equation*}
\begin{equation*}
\sum_{j \in \mathcal{G}}g_j - \sum_{k \in \mathcal{K}}\phi _{k}p_k = G_{D},
\end{equation*}
 
\begin{equation*}
\;  \underline{p}_i \leq p_i \leq \overline{p}_i, \; 0 \leq g_j \leq \overline{g}_j,
\end{equation*}
where $q$ is the vector which contains: the conventional power generation $p_{\tau}$ (whose generators are in layer $\mathcal{T}$), the gas-fired generation $p_{k}$ (whose generators are in layer $\mathcal{K}$), and the gas supply $g_{j}$ (whose generators are in layer $\mathcal{G}$). $\mathcal{L}$ is the supra-Laplacian matrix of the three-layer multiplex network (See Fig. \ref{figapp1}). $P_{D}$ is the electrical power demand, $G_{D}$ is the inelastic gas demand, and $\phi_k$ is the fuel conversion factor. $C_{\tau}^{[\mathcal{T}]}$, $C_{k}^{[\mathcal{K}]}$, and $C_{j}^{[\mathcal{G}]}$ are the cost functions of power production and gas suppliers in each layer.

\begin{figure}[t]
\includegraphics[width=3.4in]{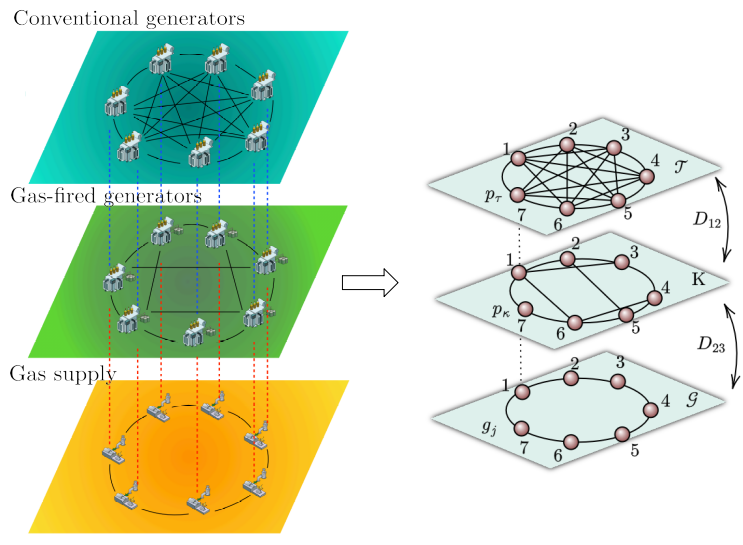}
\caption{Schematics for the three-layer micro-grid power generation and dispatch.}
\label{figapp1}
\end{figure}

The nodal time-varying gas demand $\phi_k p_k$ is determined based on the fuel consumption of natural gas-fired generators  via the fuel consumption factor. 

Therefore, the associated saddle-point flow yields
\begin{equation*}
    \begin{array}{ll}
\dot{q}&  =-\frac{\partial C (q)}{\partial q}-\mathcal{L}q-\mathcal{L}\lambda -\mu \frac{\partial g(q)}{\partial q},\\
 \dot{\lambda}   & =\mathcal{L}q,\\
 \dot{\mu}&=g(q)
\end{array}
\end{equation*}

with
\begin{equation*}
g (q) = \left(\begin{IEEEeqnarraybox*}[][c]{,c,}
\sum_{\tau \in \mathcal{T}}p_{\tau} + \sum_{k \in \mathcal{K}}p_k - P_{D}  \\
\sum_{j \in \mathcal{G}}g_j - \sum_{k \in \mathcal{K}}\phi _{k}p_k - G_{D} 
\end{IEEEeqnarraybox*}\right) .
\end{equation*}

Let us assume a cost function defined by

\begin{eqnarray*}
C(q)   =\begin{pmatrix}
		\cdots & C_{\tau }^{[\mathcal{T}]} & \cdots &	C_{\kappa }^{[ K ]} & \cdots &	C_{j}^{[\mathcal{G}]} &	\cdots 
	\end{pmatrix}^{\top}\\
   =\begin{pmatrix}
		\cdots & \frac{1}{2} p_{\tau }^{2}&	\cdots & \frac{1}{2} p_{\kappa }^{2}&	\cdots & \frac{1}{2} g_{j}^{2}& \cdots 
	\end{pmatrix}^{\top}.
\end{eqnarray*}

For this situation, the supra-Laplacian matrix reads
\begin{equation*}
\mathcal{L}=\begin{pmatrix}
		\mathcal{L}_{\mathcal{T}} & -D_{12} I_{7\times 7} & 0\\
		-D_{12} I_{7\times 7} & \mathcal{L}_{K} & -D_{23} I_{7\times 7}\\
		0 & -D_{23} I_{7\times 7} & \mathcal{L}_{\mathcal{G}}
	\end{pmatrix}, 
\end{equation*}
being $\mathcal{L}_{\mathcal{T}}$, $\mathcal{L}_{\mathcal{K}}$, and $\mathcal{L}_{\mathcal{G}}$ the Laplacian matrix for the layer $\mathcal{T}$, $\mathcal{K}$, and $\mathcal{G}$ respectively.

For the numerical implementation, we have that
\begin{align*}
	-\frac{\partial C(q)}{\partial q} & =-\begin{pmatrix}
		\cdots & p_{\tau }&	\cdots & p_{\kappa}& \cdots & g_{j}& \cdots 
	\end{pmatrix}^{\top} =-q .
\end{align*} 

Therefore, the temporal evolution of the vector $q$ is given by (with $n=7$)
\begin{align*}
	\dot{q} & =v\\
	\dot{v} & =-( I_{3n\times 3n} +\mathcal{L})v -\mathcal{L}^{2}q -g(q)\frac{\partial g(q)}{\partial q}.
\end{align*}

In Fig. \ref{figapp2} we depict the consensus dynamics for the following parameters: $D_{12} = D_{23} = D_x = 0.6$, $D_{ \mathcal{T}} = D_{\mathcal{G}} = 0.2$, $D_{K}=0.8$, $\phi_k=0.7$, $P_D= G_{D}=100$. We can observe how the nodes in each layer achieve consensus by converging to values that minimize the total cost of operating power and natural gas microgrids. Notice that, in this situation, conventional generators are converging to the same value of gas supply. Gas-fired generators are going to elevate cost value. We can observe, how the totally connected graph, is going to the same value as the less connected structure.
\begin{figure}[!h]
\includegraphics[width=3.4in]{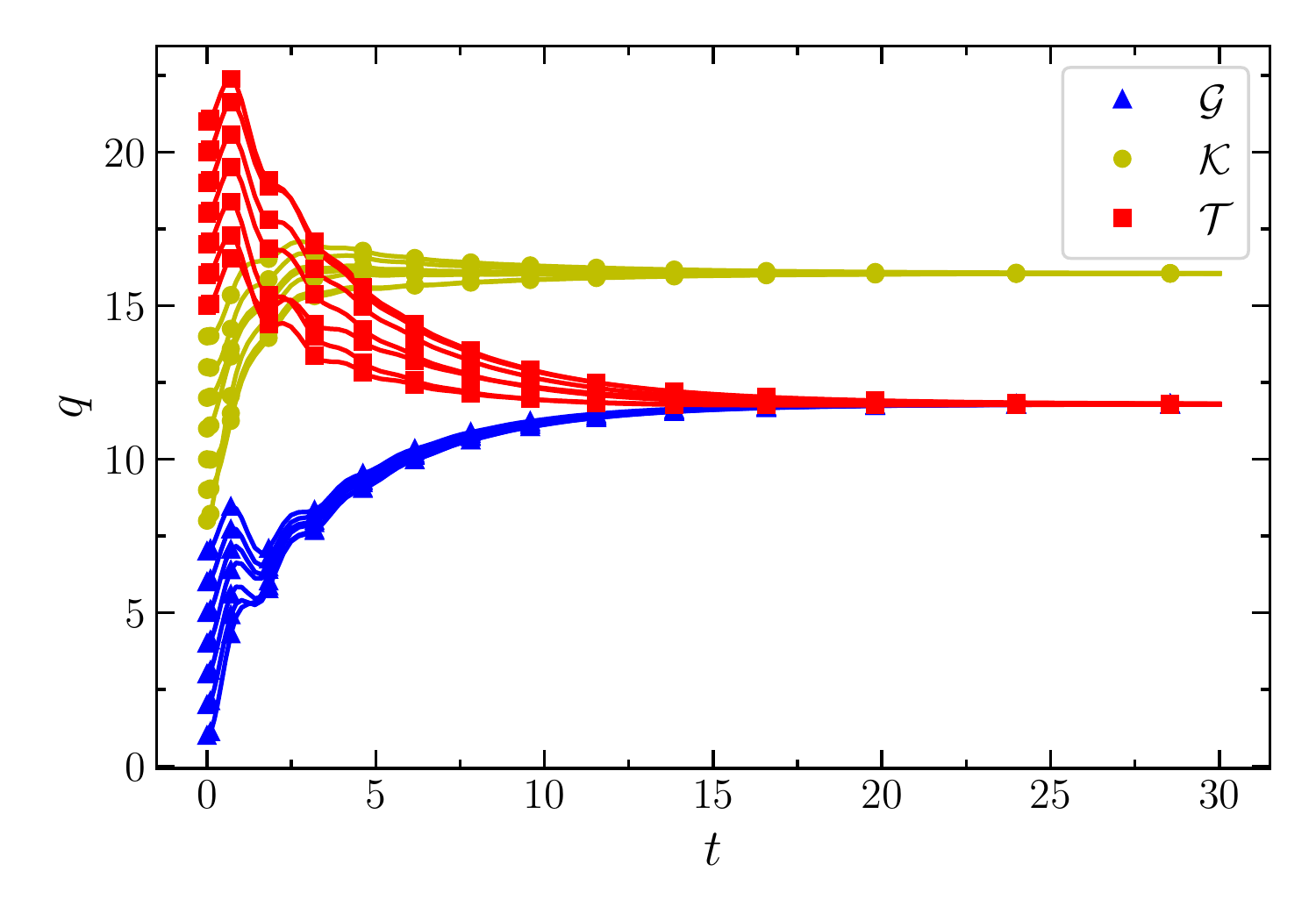}
\caption{Simulation for the coordinated dispatch for multienergy system with two-layers: energy and gas.}
\label{figapp2}
\end{figure}

\section{Conclusion}

Currently, considering recent advances in network science, statistical mechanics, and big data, it has become evident that different networked systems are part of larger structures that define network behavior. Multilayer networks have emerged as an effective tool to model these situations and understand how networks interact with other networks. Multilayer systems exhibit various phenomena seen in traditional monolayer networks, including percolation, phase transitions, diffusion, and epidemic spreading. Despite this, there is a lack of results generalizing distributed control and optimization from monolayer networks to multilayer networks. In this paper, we have obtained two algorithms for distributed optimization in a particular case of multilayer networks with a one-by-one relation between the nodes in each layer. By leveraging the relationship between diffusion and consensus dynamics and employing a control approach, we developed a distributed primal-dual saddle point algorithm that exhibits fast convergence speed and robustness. We found that intra- and interlayer diffusion constants act as control parameters for optimal consensus time. We observed that each layer achieves a local consensus before seeking a global consensus. Additionally, we discovered a critical phenomenon linked to the relationship between the consensus time and the interlayer diffusion constant. This finding complements previous work on determining diffusion dynamics for multilayer networks, which identified a phase transition related to the second eigenvalue of the supra-Laplacian matrix. Finally, we developed a distributed gradient descent algorithm for multilayer networks and observed that a time-varying positive gain plays an essential role in the convergence dynamics of the system.

Possible extensions to this work involve obtaining a diffusion operator for more general multilayer systems and applying the methodology presented here. Recent advancements in the tensor representation of multilayer networks have enabled the creation of such objects~\cite{battiston2020networks}. Additionally, recent research has developed Laplacian-like operators for hypergraphs and simplicial complexes~\cite{horak2013spectra, majhi2022dynamics}, which could be useful in developing distributed optimization algorithms for complex systems described using these general models~\cite{battiston2020networks}. Other related reaction-diffusion operators, such as the Dirac operator~\cite{calmon2023dirac}, could also be explored. Another potential extension involves using graph zeta functions to investigate diffusion dynamics in higher-order networks~\cite{chinta2015heat, saldivar2020functional}. These topics are currently being researched by the authors.


%



\appendices
\section{Proof of Theorem \ref{thm2}}
\label{Proofthm2}
\begin{proof}
First, consider the convergence analysis to the points $(y^{\star}, \lambda^{\star})$. Let $\tilde{y}=y-y^{\star}$ and $\tilde{\lambda}=\lambda -\lambda^{\star}$.  In the new variables, the saddle-point flow (\ref{saddleflow1})-(\ref{saddleflow2}) yields

\begin{equation}
\dot{\tilde{y}}=-\nabla_y \tilde{f}(y)+ \nabla_y \tilde{f}(y^{\star})-\mathcal{L}\tilde{y}-\mathcal{L}\tilde{\lambda},
\label{saddleflow3}
\end{equation} 

\begin{equation}
\dot{\tilde{\lambda}}=\mathcal{L}\tilde{y} .
\label{saddleflow4}
\end{equation}

Consider the following quadratic candidate Lyapunov function
\begin{equation}
V(\tilde{y},\tilde{\lambda})=\frac{1}{2}\tilde{y}^{T}\tilde{y}+\frac{1}{2}\tilde{\lambda}^{T}\tilde{\lambda}.
\label{lyapunov1}
\end{equation}

Then, the temporal derivative of (\ref{lyapunov1}) along the trajectories of (\ref{saddleflow1})-(\ref{saddleflow2}) is
\begin{equation}
\dot{V}(\tilde{y},\tilde{\lambda})=-\tilde{y}^{T}\nabla_y \tilde{f}(y)+\tilde{y}^{T}\nabla_y\tilde{f}(y^{\star})-\tilde{y}^{T}\mathcal{L}\tilde{y}.
\end{equation}

Based on the property that the gradient is a global under-estimator, we can conclude that
\begin{equation}
\tilde{f}(y^{\star})\geq \tilde{f}(y)+(y^{\star}-y)^{T}\nabla_y \tilde{f}(y).
\label{ineqproof}
\end{equation}

From inequality (\ref{ineqproof}), and if $\tilde{f}$ is strictly convex, we have
\begin{equation*}
-\tilde{y}^{T}\nabla_y \tilde{f}(y)+\tilde{y}^{T}\nabla_y \tilde{f}(y^{\star}) < 0.
\end{equation*}

Now, since the supra-Laplacian matrix is semi-positive definite, we obtain that
\begin{equation*}
\dot{V}(\tilde{y},\tilde{\lambda})=-\tilde{y}^{T}\nabla_y \tilde{f}(y)+\tilde{y}^{T}\nabla_y \tilde{f}(y^{\star})-\tilde{y}^{T}\mathcal{L}\tilde{y} <0.
\end{equation*}

However, $\dot{V}(\tilde{y},\tilde{\lambda})=0$ in the set $E=\{ (\tilde{y},\tilde{\lambda})|\tilde{y}=0 \}$. 
To demonstrate the global asymptotic stability of the equilibrium point $(y^{\star},\lambda ^{\star})$, we utilize LaSalle's invariance principle to prove that there are no trajectories in $E$ different from the equilibrium point $\tilde{y}=0$ and $\tilde{\lambda}=0$. Notice that if $\tilde{y}=0$ the system (\ref{saddleflow3})-(\ref{saddleflow4}) reduces to
\begin{equation*}
\dot{\tilde{y}}=-\mathcal{L}\tilde{\lambda}
\end{equation*}
\begin{equation*}
\dot{\tilde{\lambda}}=0.
\end{equation*}

It follows that since the supra-Laplacian matrix $\mathcal{L}$ is semi-positive definite and by Lemma \ref{lem1} the solution is $(0, \gamma 1_{N\cdot M})$ or $(0,0)$. According to Lemma \ref{lem1}, we can identify as $y^{\star}=x^{\star}1_{N\cdot M}$ the optimal points $(y^{\star}, \lambda^{\star})$. Furthermore, $\lambda^{\star} = \bar{\lambda}+\gamma 1_{N\cdot M}$, where $\gamma \in \mathbb{R}$ and $\bar{\lambda}\perp 1_{N\cdot M}$ satisfies
\begin{equation*}
\mathcal{L}\bar{\lambda}+\nabla_y \tilde{f}(y^{\star})=0_{N\cdot M}.
\end{equation*}

By examining the dual dynamics (\ref{saddleflow4}), we can observe that $1_{N\cdot M}\dot{\lambda}=0$. Hence, we have $\text{average}(\lambda (t))=\text{average}(\lambda_{0})$ $\forall t \geq 0$. Consequently, the convergence to a single saddle-point $(y^{\star}, \lambda ^{\star})$ of the flow (\ref{saddleflow3})-(\ref{saddleflow4}) is guaranteed, which satisfies $y^{\star}=x^{\star}1_{N\cdot M}$ and $\lambda^{\star} = \bar{\lambda}+ \text{average}(\lambda_{0}) 1_{N\cdot M}$.
\end{proof}

\section{Proof of Theorem \ref{thm3}}
\label{Proofthm3}
\begin{proof}
    First, we consider the analysis of the state $y$ boundedness by considering the following Lyapunov function
\begin{equation}
V(y)=\frac{1}{2}\| y - y^{*}\| _{2}^{2} .
\label{lyap51}
\end{equation}

The derivative of (\ref{lyap51}) along the trajectories of (\ref{distgflow2}) is 
\begin{equation*}
\dot{V}(y)=-\varsigma(t)(y-y^{*})^{T}\nabla_y \tilde{f}(y)-(y-y^{*})^{T}\mathcal{L}y.
\end{equation*}

Taking advantage of the of under estimator property (\ref{ineqproof}), and that $\mathcal{L}y = \mathcal{L}x^{*}1_{N\cdot M}=0$, the last expression yields
\begin{equation*}
\dot{V}(y)=-\varsigma(t)(\tilde{f}(y)-\tilde{f}(y^{*}))-(y-y^{*})^{T}\mathcal{L}(y-y^{*}).
\end{equation*}

Thus, $\dot{V}(y)\leq 0$, and then, the state $y$ is bounded. Therefore, the gap between the state $y$ and the optimizer is given by
\begin{equation}
\| y(t) - y^{*}\| \leq \| y _{0} - y^{*}\| .
\label{bound51}
\end{equation}

Within this result, we shall evaluate the asymptotic consensus and optimality via the following coordinate transformation
\begin{equation}
\begin{matrix}
\left[\begin{IEEEeqnarraybox*}[][c]{,c,}
\psi \\
\eta %
\end{IEEEeqnarraybox*}\right]
\end{matrix} =
\begin{matrix}
\left[\begin{IEEEeqnarraybox*}[][c]{,c,}
U ^{T} \\
1_{N\cdot M}^{T}/(N\cdot M) %
\end{IEEEeqnarraybox*}\right] y.
\end{matrix}
\end{equation}

Being $U$ an orthogonal matrix such that
\begin{equation}
U^{T}\mathcal{L}U = -\text{diag}(\lambda _{2},..., \lambda _{N\cdot M}) .
\end{equation}

The dynamics of the $\psi (t)$ function reads
\begin{equation}
\dot{\psi}(t) = - \text{diag}(\lambda _{2},..., \lambda _{N\cdot M}) \psi (t) + \xi (t),
\label{psi111}
\end{equation}
being $\xi (t)=-\varsigma (t) U^{T}\nabla_y \tilde{f}(y)$.

The input signal (or source term) in (\ref{psi111}) is bounded due to (\ref{bound51}) and it is decaying to zero by the properties established for $\varsigma (t)$. The system governed by (\ref{psi111}) is an exponentially stable linear system under the influence of the input $\xi(t)$. Therefore, $\psi (t) \rightarrow 0$ as $t \rightarrow \infty$, and $y(t)\rightarrow \eta $, that is, $\psi (t\rightarrow \infty)\in \text{span}(1_{N\cdot M})$, achieving asymptotically the consensus.

With the remaining coordinate $\eta = \text{average}(y)$, we shall study the asymptotic optimality. The dynamics driving this coordinate are given by

\begin{eqnarray}
\dot{\eta} &=& -\frac{\varsigma (t)}{N\cdot M}1_{N\cdot M}^{T}\nabla_y \tilde{f}(y) = -\frac{\varsigma (t)}{N\cdot M} \sum _{i = 1}^{N\cdot M}\nabla_{y_{i}} f _{i}(y_{i}) \nonumber
\\
&=&  -\frac{\varsigma (t)}{N\cdot M} \sum _{i = 1}^{N\cdot M}\frac{\partial f _{i}}{\partial y_{i}}((U\psi)_{i}+\eta) .
\end{eqnarray}

Considering the following time scaling
\begin{equation*}
\tau = \int _{0}^{t}\varsigma (t')dt'\, \Rightarrow \, \frac{d\tau}{dt}= \varsigma (t),
\end{equation*}
we have $\tau$ increasing in $t$ monotonically as $\varsigma (t) > 0$. Conversely, the integral persistence condition maps the interval $t\in [0, \infty]$ to $\tau=[0,\infty]$. Then the time coordinate transformation is invertible. The equations of motion that control the evolution of the coordinate $\eta$ in $\tau$-time-scale are expressed as follows 
\begin{equation}
\frac{d}{d\tau}\eta = -\frac{1}{n} \sum _{\alpha=1}^{M}\sum _{i=1}^{N}\frac{\partial f _{i}^{[\alpha]}}{\partial y_{i}^{[\alpha ]}}((U\psi)_{i}+\eta) .
\end{equation}

With this in mind, let us consider the following Lyapunov function
\begin{equation}
W(\eta) = \frac{n}{2}(\eta -\eta ^{*})^{2} ,
\label{lyapeta51}
\end{equation}
being $\eta ^{*}=x^{*}$ the minimum of $f(x)=\sum _{\alpha=1}^{M}\sum _{i=1}^{N}f_{i}^{[\alpha]}(x)$. 

The temporal-scaled derivative of (\ref{lyapeta51}) is

\begin{eqnarray*}
\frac{d}{d\tau}W(\eta) = -(\eta - \eta ^{*})\sum _{\alpha=1}^{M}\sum _{i=1}^{N}\frac{\partial f_{i}^{[\alpha]}}{\partial y_{i}^{[\alpha]}}(\eta)  \nonumber
\\
+ (\eta ^{*}-\eta)\sum _{\alpha=1}^{M}\sum _{i=1}^{N}\left[ \frac{\partial f_{i}^{[\alpha]}}{\partial y_{i}^{[\alpha]}}((U\psi)_{i}^{[\alpha]}+\eta) \right. - \left. \frac{\partial f_{i}^{[\alpha]}}{\partial y_{i}^{[\alpha]}}(\eta) \right]. 
\end{eqnarray*}

By the under estimator property, the first term of $\frac{d}{d\tau}W(\eta)$ is upper-bounded by $f(\eta ^{*})-f(\eta)$. The second term of $\frac{d}{d\tau}W(\eta)$ can be upper-bounded as follows
\begin{eqnarray*}
(\eta ^{*}-\eta)\sum _{\alpha=1}^{M}\sum _{i=1}^{N}\left[ \frac{\partial f_{i}^{[\alpha]}}{\partial y_{i}^{[\alpha]}}((U\psi)_{i}^{[\alpha]}+\eta) - \frac{\partial f_{i}^{[\alpha]}}{\partial y_{i}^{[\alpha]}}(\eta) \right] \nonumber\\ \leq | \eta ^{*}-\eta | \sum _{\alpha=1}^{M}\sum _{i=1}^{N} L_{i}^{[\alpha]}|(U\psi)_{i}^{[\alpha]} |.
\label{ineq511}
\end{eqnarray*}

It can be done due to the boundedness of the state $y$. Notice that $L_{i}^{[\alpha]}$ is the greatest Lipschitz constant of $\frac{\partial f _{i}^{[\alpha]}}{\partial y}$ in inequality (\ref{ineq511}), estimated in $\{ y \in \mathbb{R}^{N\cdot M} | \| y - y^{*}\| \leq \| y_{0} - y^{*}\| \}$, which is the invariant set defined by
\begin{equation}
L_{i}^{[\alpha]}=\max _{\| y - y^{*}\| \leq \| y_{0} - y^{*}\|} \text{Hess}f_{i}^{[\alpha]}(y) .
\end{equation}

Due to the asymptotic consensus, for each $\epsilon \geq 0$, there is a $\tau _{\epsilon}$ so that for all $\tau \geq \tau _{\epsilon}$ we have
\begin{equation*}
\sum _{\alpha=1}^{M}\sum _{i=1}^{N} L_{i}^{[\alpha]}|(U\psi)_{i}^{[\alpha]}|\leq \epsilon.
\end{equation*}

That is, for each $\tau \geq \tau _{\epsilon}$

\begin{equation}
\frac{d}{d\tau}W(\eta) \leq f(\eta ^{*})-f(\eta)+\epsilon |\eta ^{*}-\eta | \leq 0.
\end{equation}

Since $\tilde{f}$ (and also $f$) is radially unbound and convex, we have that there is a $|\eta _{\epsilon}|$ such that the temporal-scaled derivative of (\ref{lyapeta51}) is strictly negative for sufficiently large $|\eta ^{*}-\eta | \geq |\eta _{\epsilon}|$ and therefore $W(\eta (\tau \rightarrow \infty ))\leq W(\eta _{\epsilon})$. Additionally, $\eta _{\epsilon} \rightarrow 0$ as $\epsilon \rightarrow 0$. For any given $\tilde{\epsilon}\geq 0$, there exists a sufficiently small $\epsilon$ such that for $\tau \geq \tau _{\epsilon}$, $W(\eta (\tau))\leq \tilde{\epsilon}$. This implies that $W(\eta (\tau))$ approaches $0$ as $\eta (\tau)$ converges to the optimal value $\eta ^{*}$ in the limit of $\tau \rightarrow \infty$. The same convergence occurs for $\eta (t)$ as $t\rightarrow \infty$.
\end{proof}
 \section*{Acknowledgments}
This paper was partially supported by Minciencias Grant number CT
542-2020, \emph{Programa de Investigación en Tecnologías Emergentes para Microrredes Electricas Inteligentes con Alta Penetración de Energías Renovables} and the \emph{VIII Convocatoria para el Desarrollo y Fortalecimiento de los Grupos de Investigaci\'on en Uniminuto} with code C119-173, the \emph{Convocatoria de investigación para prototipado de tecnologías que promueven el cuidado o la restauración del medioambiente} with code CPT123-200-5220, and Industrial Engineering Program from the Corporaci\'on Universitaria Minuto de Dios (Uniminuto, Colombia). We thank the Engineering and Physical Sciences Research Council (EPSRC) (Grants No. EP/R513143/1 and No. EP/T517793/1) for financial support.




\bibliographystyle{IEEEtran}
\bibliography{Multiplexbib}

\begin{thebibliography}{10}
\providecommand{\url}[1]{#1}
\csname url@samestyle\endcsname
\providecommand{\newblock}{\relax}
\providecommand{\bibinfo}[2]{#2}
\providecommand{\BIBentrySTDinterwordspacing}{\spaceskip=0pt\relax}
\providecommand{\BIBentryALTinterwordstretchfactor}{4}
\providecommand{\BIBentryALTinterwordspacing}{\spaceskip=\fontdimen2\font plus
\BIBentryALTinterwordstretchfactor\fontdimen3\font minus
  \fontdimen4\font\relax}
\providecommand{\BIBforeignlanguage}[2]{{%
\expandafter\ifx\csname l@#1\endcsname\relax
\typeout{** WARNING: IEEEtran.bst: No hyphenation pattern has been}%
\typeout{** loaded for the language `#1'. Using the pattern for}%
\typeout{** the default language instead.}%
\else
\language=\csname l@#1\endcsname
\fi
#2}}
\providecommand{\BIBdecl}{\relax}
\BIBdecl

\bibitem{barabasi2016network}
A.~Barab{\'a}si, \emph{Network Science}.\hskip 1em plus 0.5em minus 0.4em\relax
  Cambridge University Press, Cambridge, USA, 2016.

\bibitem{newman2010networks}
M.~Newman, \emph{Networks: An Introduction}.\hskip 1em plus 0.5em minus
  0.4em\relax Oxford, UK: Oxford University Press, 2010.

\bibitem{albert2002statistical}
R.~Albert and A.-L. Barab{\'a}si, ``Statistical mechanics of complex
  networks,'' \emph{Reviews of modern physics}, vol.~74, no.~1, p.~47, 2002.

\bibitem{dorogovtsev2003evolution}
S.~N. Dorogovtsev and J.~F. Mendes, \emph{Evolution of networks: From
  biological nets to the Internet and WWW}.\hskip 1em plus 0.5em minus
  0.4em\relax New York: Oxford University Press, 2003.

\bibitem{bianconi2018multilayer}
G.~Bianconi, \emph{Multilayer networks: structure and function}.\hskip 1em plus
  0.5em minus 0.4em\relax Oxford University Press, Oxford, UK, 2018.

\bibitem{kivela2014multilayer}
M.~Kivela, A.~Arenas, M.~Barthelemy, J.~P. Gleeson, Y.~Moreno, and M.~A.
  Porter, ``Multilayer networks,'' \emph{Journal of complex networks}, vol.~2,
  no.~3, pp. 203--271, 2014.

\bibitem{boccaletti2014structure}
S.~Boccaletti, G.~Bianconi, R.~Criado, C.~I. Del~Genio, J.~G{\'o}mez-Gardenes,
  M.~Romance, I.~Sendina-Nadal, Z.~Wang, and M.~Zanin, ``The structure and
  dynamics of multilayer networks,'' \emph{Physics reports}, vol. 544, no.~1,
  pp. 1--122, 2014.

\bibitem{vaiana2020multilayer}
M.~Vaiana and S.~F. Muldoon, ``Multilayer brain networks,'' \emph{Journal of
  Nonlinear Science}, vol.~30, no.~5, pp. 2147--2169, 2020.

\bibitem{kiani2021networks}
N.~A. Kiani, D.~Gomez-Cabrero, and G.~Bianconi, \emph{Networks of Networks in
  Biology: Concepts, Tools and Applications}.\hskip 1em plus 0.5em minus
  0.4em\relax Cambridge University Press, 2021.

\bibitem{pilosof2017multilayer}
S.~Pilosof, M.~A. Porter, M.~Pascual, and S.~K{\'e}fi, ``The multilayer nature
  of ecological networks,'' \emph{Nature Ecology \& Evolution}, vol.~1, no.~4,
  p. 0101, 2017.

\bibitem{du2016physics}
W.-B. Du, X.-L. Zhou, M.~Jusup, and Z.~Wang, ``Physics of transportation:
  Towards optimal capacity using the multilayer network framework,''
  \emph{Scientific reports}, vol.~6, no.~1, pp. 1--8, 2016.

\bibitem{toro2021multiplex}
V.~Toro, E.~Mojica-Nava, and N.~Rakoto-Ravalontsalama, ``Multiplex centrality
  measurements applied to islanded microgrids,'' \emph{International Journal of
  Control, Automation and Systems}, vol.~19, no.~1, pp. 449--458, 2021.

\bibitem{santoro2020optimal}
A.~Santoro and V.~Nicosia, ``Optimal percolation in correlated multilayer
  networks with overlap,'' \emph{Physical Review Research}, vol.~2, no.~3, p.
  033122, 2020.

\bibitem{turalska2019cascading}
M.~Turalska, K.~Burghardt, M.~Rohden, A.~Swami, and R.~M. D'Souza, ``Cascading
  failures in scale-free interdependent networks,'' \emph{Physical Review E},
  vol.~99, no.~3, p. 032308, 2019.

\bibitem{perc2019diffusion}
M.~Perc, ``Diffusion dynamics and information spreading in multilayer networks:
  An overview,'' \emph{The European Physical Journal Special Topics}, vol. 228,
  no.~11, pp. 2351--2355, 2019.

\bibitem{shang2020resilient}
Y.~Shang, ``Resilient consensus for robust multiplex networks with asymmetric
  confidence intervals,'' \emph{IEEE Transactions on Network Science and
  Engineering}, vol.~8, no.~1, pp. 65--74, 2020.

\bibitem{guo2021adaptive}
Z.~Guo, M.~Lian, S.~Wen, and T.~Huang, ``An adaptive multi-agent system with
  duplex control laws for distributed resource allocation,'' \emph{IEEE
  Transactions on Network Science and Engineering}, vol.~9, no.~2, pp.
  389--400, 2021.

\bibitem{salehi2015spreading}
M.~Salehi, R.~Sharma, M.~Marzolla, M.~Magnani, P.~Siyari, and D.~Montesi,
  ``Spreading processes in multilayer networks,'' \emph{IEEE Transactions on
  Network Science and Engineering}, vol.~2, no.~2, pp. 65--83, 2015.

\bibitem{yang2019survey}
T.~Yang, X.~Yi, J.~Wu, Y.~Yuan, D.~Wu, Z.~Meng, Y.~Hong, H.~Wang, Z.~Lin, and
  K.~H. Johansson, ``A survey of distributed optimization,'' \emph{Annual
  Reviews in Control}, vol.~47, pp. 278--305, 2019.

\bibitem{nedic2009distributed}
A.~Nedic and A.~Ozdaglar, ``Distributed subgradient methods for multi-agent
  optimization,'' \emph{IEEE Transactions on Automatic Control}, vol.~54,
  no.~1, pp. 48--61, 2009.

\bibitem{nedic2018distributed}
A.~Nedic and J.~Liu, ``Distributed optimization for control,'' \emph{Annual
  Review of Control, Robotics, and Autonomous Systems}, vol.~1, pp. 77--103,
  2018.

\bibitem{Elia2011control}
J.~Wang and N.~Elia, ``A control perspective for centralized and distributed
  convex optimization,'' in \emph{2011 50th IEEE conference on decision and
  control and European control conference}.\hskip 1em plus 0.5em minus
  0.4em\relax IEEE, 2011, pp. 3800--3805.

\bibitem{kia2015dynamic}
S.~S. Kia, J.~Cort{\'e}s, and S.~Martinez, ``Dynamic average consensus under
  limited control authority and privacy requirements,'' \emph{International
  Journal of Robust and Nonlinear Control}, vol.~25, no.~13, pp. 1941--1966,
  2015.

\bibitem{lin2016distributed}
P.~Lin, W.~Ren, and J.~A. Farrell, ``Distributed continuous-time optimization:
  nonuniform gradient gains, finite-time convergence, and convex constraint
  set,'' \emph{IEEE Transactions on Automatic Control}, vol.~62, no.~5, pp.
  2239--2253, 2016.

\bibitem{yang2016multi}
S.~Yang, Q.~Liu, and J.~Wang, ``A multi-agent system with a
  proportional-integral protocol for distributed constrained optimization,''
  \emph{IEEE Transactions on Automatic Control}, vol.~62, no.~7, pp.
  3461--3467, 2016.

\bibitem{wu2023distributed}
C.~Wu, H.~Fang, X.~Zeng, Q.~Yang, Y.~Wei, and J.~Chen, ``Distributed
  continuous-time algorithm for time-varying optimization with affine formation
  constraints,'' \emph{IEEE Transactions on Automatic Control}, vol.~68, no.~4,
  pp. 2615--2622, 2023.

\bibitem{ma2019novel}
L.~Ma and W.~Bian, ``A novel multiagent neurodynamic approach to constrained
  distributed convex optimization,'' \emph{IEEE Transactions on Cybernetics},
  vol.~51, no.~3, pp. 1322--1333, 2021.

\bibitem{cherukuri2017saddle}
A.~Cherukuri, B.~Gharesifard, and J.~Cortes, ``Saddle-point dynamics:
  conditions for asymptotic stability of saddle points,'' \emph{SIAM Journal on
  Control and Optimization}, vol.~55, no.~1, pp. 486--511, 2017.

\bibitem{feijer2010stability}
D.~Feijer and F.~Paganini, ``Stability of primal--dual gradient dynamics and
  applications to network optimization,'' \emph{Automatica}, vol.~46, no.~12,
  pp. 1974--1981, 2010.

\bibitem{colombino2019online}
M.~Colombino, E.~Dall’Anese, and A.~Bernstein, ``Online optimization as a
  feedback controller: Stability and tracking,'' \emph{IEEE Transactions on
  Control of Network Systems}, vol.~7, no.~1, pp. 422--432, 2019.

\bibitem{hauswirth2021optimization}
A.~Hauswirth, S.~Bolognani, G.~Hug, and F.~D{\"o}rfler, ``Optimization
  algorithms as robust feedback controllers,'' \emph{arXiv preprint
  arXiv:2103.11329}, 2021.

\bibitem{feedback2022}
D.~Krishnamoorthy and S.~Skogestad, ``Real-time optimization as a feedback
  control problem-a review,'' \emph{Computers \& Chemical Engineering}, p.
  107723, 2022.

\bibitem{nedic2017achieving}
A.~Nedic, A.~Olshevsky, and W.~Shi, ``Achieving geometric convergence for
  distributed optimization over time-varying graphs,'' \emph{SIAM Journal on
  Optimization}, vol.~27, no.~4, pp. 2597--2633, 2017.

\bibitem{mancarella2014mes}
P.~Mancarella, ``Mes (multi-energy systems): An overview of concepts and
  evaluation models,'' \emph{Energy}, vol.~65, pp. 1--17, 2014.

\bibitem{guelpa2019towards}
E.~Guelpa, A.~Bischi, V.~Verda, M.~Chertkov, and H.~Lund, ``Towards future
  infrastructures for sustainable multi-energy systems: A review,''
  \emph{Energy}, vol. 184, pp. 2--21, 2019.

\bibitem{chertkov2020multienergy}
M.~Chertkov and G.~Andersson, ``Multienergy systems,'' \emph{Proceedings of the
  IEEE}, vol. 108, no.~9, pp. 1387--1391, 2020.

\bibitem{bertsekas2009convex}
D.~Bertsekas, \emph{Convex optimization theory}.\hskip 1em plus 0.5em minus
  0.4em\relax Athena Scientific, 2009, vol.~1.

\bibitem{dorfler2017distributed}
\BIBentryALTinterwordspacing
F.~Dorfler. (2018) Distributed consensus-based optimization. advanced topics in
  control 2018: Distributed systems \& control. Accessed on April 19, 2023.
  [Online]. Available:
  \url{http://people.ee.ethz.ch/~floriand/docs/Teaching/ATIC_2018/Optimization_Lecture.pdf}
\BIBentrySTDinterwordspacing

\bibitem{buldyrev2010catastrophic}
S.~V. Buldyrev, R.~Parshani, G.~Paul, H.~E. Stanley, and S.~Havlin,
  ``Catastrophic cascade of failures in interdependent networks,''
  \emph{Nature}, vol. 464, no. 7291, pp. 1025--1028, 2010.

\bibitem{gomez2013diffusion}
S.~Gomez, A.~Diaz-Guilera, J.~Gomez-Gardenes, C.~J. Perez-Vicente, Y.~Moreno,
  and A.~Arenas, ``Diffusion dynamics on multiplex networks,'' \emph{Physical
  review letters}, vol. 110, no.~2, p. 028701, 2013.

\bibitem{biggs1993algebraic}
N.~Biggs, N.~L. Biggs, and B.~Norman, \emph{Algebraic graph theory}.\hskip 1em
  plus 0.5em minus 0.4em\relax Cambridge university press, 1993, no.~67.

\bibitem{curtis1991dirichlet}
E.~B. Curtis and J.~A. Morrow, ``The dirichlet to neumann map for a resistor
  network,'' \emph{SIAM Journal on Applied Mathematics}, vol.~51, no.~4, pp.
  1011--1029, 1991.

\bibitem{gilbert2016diffuse}
A.~C. Gilbert, J.~G. Hoskins, and J.~C. Schotland, ``Diffuse scattering on
  graphs,'' \emph{Linear Algebra and its Applications}, vol. 496, pp. 1--35,
  2016.

\bibitem{guo2024neurodynamic}
L.~Guo, I.~Korovin, S.~Gorbachev, X.~Shi, N.~Gorbacheva, and J.~Cao,
  ``Neurodynamic approaches for multi-agent distributed optimization,''
  \emph{Neural Networks}, vol. 169, pp. 673--684, 2024.

\bibitem{battiston2020networks}
F.~Battiston, G.~Cencetti, I.~Iacopini, V.~Latora, M.~Lucas, A.~Patania, J.-G.
  Young, and G.~Petri, ``Networks beyond pairwise interactions: structure and
  dynamics,'' \emph{Physics Reports}, vol. 874, pp. 1--92, 2020.

\bibitem{horak2013spectra}
D.~Horak and J.~Jost, ``Spectra of combinatorial laplace operators on
  simplicial complexes,'' \emph{Advances in Mathematics}, vol. 244, pp.
  303--336, 2013.

\bibitem{majhi2022dynamics}
S.~Majhi, M.~Perc, and D.~Ghosh, ``Dynamics on higher-order networks: A
  review,'' \emph{Journal of the Royal Society Interface}, vol.~19, no. 188, p.
  20220043, 2022.

\bibitem{calmon2023dirac}
L.~Calmon, M.~T. Schaub, and G.~Bianconi, ``Dirac signal processing of
  higher-order topological signals,'' \emph{arXiv preprint arXiv:2301.10137},
  2023.

\bibitem{chinta2015heat}
G.~Chinta, J.~Jorgenson, and A.~Karlsson, ``Heat kernels on regular graphs and
  generalized ihara zeta function formulas,'' \emph{Monatshefte f{\"u}r
  Mathematik}, vol. 178, no.~2, pp. 171--190, 2015.

\bibitem{saldivar2020functional}
A.~Saldivar, N.~F. Svaiter, and C.~A. Zarro, ``Functional equations for
  regularized zeta-functions and diffusion processes,'' \emph{Journal of
  Physics A: Mathematical and Theoretical}, vol.~53, no.~23, p. 235205, 2020.

\end{thebibliography}

%

%

\begin{IEEEbiographynophoto}{Christian David Rodr\'iguez-Camargo}
Received a B.S. degree in physics from Universidad Nacional de Colombia in 2014, an M.Sc. degree in physics from the Brazilian Center for Research in Physics (CBPF) in 2016. He does research in relativistic quantum information, quantum biology, quantum field theory in curved spaces, smart grits and statistical field theory in network science within researching collaboration at The Atomic, Molecular, Optical and Positron Physics (AMOPP) group of the Department of Physics and Astronomy from the University College London and Programa de Investigaci\'on sobre Adquisici\'on y An\'alisis de Se\~nales (PAAS-UN) from Universidad Nacional de Colombia.
\end{IEEEbiographynophoto}

\begin{IEEEbiographynophoto}{Andr\'es F. Urquijo-Rodr\'iguez}
Received a B.S. and M.Sc. degree in physics from Universidad Nacional de Colombia, and he is currently pursuing his Ph.D. from Universidad Nacional de Colomba. He does research in network science, quantum information and condensed matter physics in low dimensional systems. He is an assistant professor at Corporación Universitaria Minuto de Dios, Bogotá, Colombia.
\end{IEEEbiographynophoto}

\begin{IEEEbiographynophoto}{Eduardo Mojica-Nava}
He received a B.S. degree in Electronics Engineering from Universidad Industrial de Santander in 2002, an M.Sc. degree in Electronics and Computer Science Engineering from Universidad de Los Andes, and a Ph.D. degree in Automatique et Informatique Industrielle from École des Mines de Nantes, Nantes, France and also Universidad de Los Andes in 2010. From 2011 to 2012, he was a Post-Doctoral Researcher at Universidad de Los Andes. He has been visiting professor at Université de Mons, Mons, Belgium, and Politecnico de Milano, Milan, Italy. Currently, he is a full professor at Universidad Nacional de Colombia, Bogotá, Colombia. 
\end{IEEEbiographynophoto}







\end{document}